\documentclass{amsart}
\usepackage[utf8]{inputenc}
\usepackage[top=1in, bottom=1in]{geometry}
\usepackage{graphicx,subfigure}
\usepackage[mathscr]{eucal}

\usepackage[colorlinks=true]{hyperref}
\usepackage{enumitem}

\usepackage{cases}
\usepackage{amsfonts,amssymb,bbm}
\newcommand{\Z}{\mathbb Z}
\newcommand{\R}{\mathbb R}
\newcommand{\nonnegR}{\R^+}
\newcommand{\positR}{(0,\infty)}

\newcommand{\C}{\mathbb C}
\newcommand{\N}{\mathbb N}
\newcommand{\E}{\mathbb E}

\newcommand{\ii}{\boldsymbol{i}}

\newcommand{\K}{\mathcal K}

\renewcommand{\P}{\mathbb P}

\newcommand{\tS}{\tilde S}
\newcommand{\bS}{\underline{S}}
\newcommand{\ulmu}{\underline{\mu}}

\newcommand{\tGamma}{\tilde{\Gamma}}

\newcommand{\Xc}{\mathcal X}

\newcommand{\Yc}{\mathcal Y}

\newcommand{\Gammasf}{\mathsf{\Gamma}}

\newcommand{\ind}{\mathbbm{1}}

\newcommand{\gNr}{\mathcal N_\R}

\newcommand{\dd}{\,\mathrm d}

\newcommand{\as}{\quad\mathrm{a.s}}
\DeclareMathOperator*{\tr}{tr}

\DeclareMathOperator*{\supp}{supp}
\DeclareMathOperator*{\diag}{diag}
\DeclareMathOperator*{\var}{Var}
\DeclareMathOperator*{\cov}{Cov}
\DeclareMathOperator*{\esssup}{ess\,sup}

\newcommand{\tran}{\top}
\newcommand{\lmax}{\lambda_{\mathrm{max}}}
\newcommand{\cvweak}{\xrightarrow{\mathcal D}}
\newcommand{\cvprob}{\xrightarrow{\mathcal P}}
\usepackage{amsthm}
\newtheorem{Th}{Theorem}[section]
\newtheorem{lemma}[Th]{Lemma}
\newtheorem{prop}[Th]{Proposition}
\newtheorem{corol}[Th]{Corollary}

\theoremstyle{definition}
\newtheorem{defi}{Definition}
\newtheorem{defi-prop}[Th]{Definition-Proposition}
\newtheorem{example}{Example}[section]
\newtheorem*{example*}{Example}
\theoremstyle{remark}
\newtheorem*{Rq*}{Remark}
\newtheorem{Rq}{Remark}

\title[Unbounded largest eigenvalue of large sample covariance matrices]{Unbounded largest eigenvalue of large sample covariance matrices: Asymptotics, fluctuations and applications}
\author[F. Merlevède]{Florence Merlevède}
\address{Florence Merlevède, Laboratoire d'Analyse et de Mathématiques Appliquées (UMR 8050) 
Université Paris-Est Marne-La-Vallée, 5, boulevard Descartes,
Champs sur Marne, 77454 Marne-La-Vallée Cedex 2, France}
\email{florence.merlevede@u-pem.fr}

\author[J. Najim]{Jamal Najim}
\address{Jamal Najim, Laboratoire d'Informatique Gaspard Monge (UMR 8049)
 Université Paris-Est Marne-La-Vallée, 5, boulevard Descartes,
Champs sur Marne, 77454 Marne-La-Vallée Cedex 2, France}
\email{najim@univ-mlv.fr}

\author[P. Tian]{Peng Tian}
\address{Peng Tian, Laboratoire d'Informatique Gaspard Monge (UMR 8049) et Laboratoire d'Analyse et de Mathématiques Appliquées (UMR 8050)
Université Paris-Est Marne-La-Vallée, 5, boulevard Descartes,
Champs sur Marne, 77454 Marne-La-Vallée Cedex 2, France
}
\email{tianpeng83@gmail.com}

\thanks{The authors gratefully acknowledge the support by Labex BÉZOUT and Grant ANR-17-CE40-0003 HIDITSA}

\keywords{Large sample covariance matrices, largest eigenvalue, long memory stationary processes}
\subjclass[2010]{Primary 15B52, Secondary 15A18, 60B20, 60G10, 60G15}

\date{\today}
\begin{document}
\maketitle

\begin{abstract}

Given a large sample covariance matrix
$
S_N=\frac 1n\Gamma_N^{1/2}Z_N Z_N^*\Gamma_N^{1/2}\, ,
$
where $Z_N$ is a $N\times n$ matrix with i.i.d. centered entries, and $\Gamma_N$ is a $N\times N$ deterministic Hermitian positive semidefinite matrix, we study the location and fluctuations of $\lambda_{\max}(S_N)$, the largest eigenvalue of $S_N$ as $N,n\to\infty$ and $Nn^{-1} \to r\in\positR$ in the case where the empirical distribution $\mu^{\Gamma_N}$ of eigenvalues of $\Gamma_N$ is tight (in $N$) and $\lmax(\Gamma_N)$ goes to $+\infty$. These conditions are in particular met when $\mu^{\Gamma_N}$ weakly converges to a probability measure with unbounded support on $\nonnegR$. 

We prove that asymptotically $\lambda_{\max}(S_N)\sim \lambda_{\max}(\Gamma_N)$. Moreover when the $\Gamma_N$'s are block-diagonal, and the following {\em spectral gap condition} is assumed:
$$
\limsup_{N\to\infty} \frac{\lambda_2(\Gamma_N)}{\lambda_{\max}(\Gamma_N)}<1,
$$
where $\lambda_2(\Gamma_N)$ is the second largest eigenvalue of $\Gamma_N$, we prove Gaussian fluctuations for $\frac{\lambda_{\max}(S_N)}{\lambda_{\max}(\Gamma_N)}$ at the scale $\sqrt{n}$.

In the particular case where $Z_N$ has i.i.d. Gaussian entries and $\Gamma_N$ is the $N\times N$ autocovariance matrix of a long memory Gaussian stationary process $(\Xc_t)_{t\in\Z}$, the columns of $\Gamma_N^{1/2} Z_N$ can be considered as $n$ i.i.d. samples of the random vector $(\Xc_1,\dots,\Xc_N)^\tran$. We then prove that $\Gamma_N$ is similar to a diagonal matrix which satisfies all the required assumptions of our theorems, hence our results apply to this case.
\end{abstract}

\section{Introduction}
\label{sec:intr}
\subsection*{The model.} 
In this paper we consider the following model of sample covariance matrix
\begin{equation}
    S_N=\frac{1}{n}\Gamma_N^{1/2}Z_NZ_N^*\Gamma_N^{1/2} \label{eq:model_main}
\end{equation}
where $Z_N=\left(Z_{i,j}^{(N)}\right)$ is a $N\times n$ matrix whose entries $Z_{i,j}^{(N)}$ are real or complex random variables identically distributed (i.d.) for all $i, j, N$ and independent across $i, j$ for each $N$, satisfying
\begin{equation}\label{eq:moment-conditions}
\E Z_{i,j}^{(N)}=0, \quad \E|Z_{i,j}^{(N)}|^2=1 \quad\text{ and } \quad\E|Z_{i,j}^{(N)}|^4<\infty\,,
\end{equation}
and $\Gamma_N$ is a $N\times N$ deterministic Hermitian positive semidefinite matrix with eigenvalues 
$$0\ \le\  \lambda_N(\Gamma_N)\ \le\ \cdots\ \le\ \lambda_1(\Gamma_N):=\lambda_{\max}(\Gamma_N)\, .
$$ We consider the case where $\lmax(\Gamma_N)$ goes to infinity as $N\to\infty$ while the empirical spectral distribution (ESD) $\mu^{\Gamma_N}$ associated with $\Gamma_N$,
$$
\mu^{\Gamma_N}:=\frac{1}{N}\sum_{k=1}^N\delta_{\lambda_k(\Gamma_N)}
\, , 
$$
forms a tight sequence of probabilities on $\nonnegR:=[0,\infty)$. These conditions encompass the important case where $\mu^{\Gamma_N}$ converges to a limiting distribution with unbounded support on $\R^+$.

In this context, our aim is to study the location and fluctuations of the largest eigenvalue $\lmax(S_N)$ in the asymptotic regime where 
\begin{equation}\label{eq:asymptotic}
N,n\to\infty\qquad \textrm{and}\qquad \frac Nn \to r\in\positR\ .
\end{equation}
The regime \eqref{eq:asymptotic} will be simply refered to as $N,n\to\infty$ in the sequel.

The model $S_N$ defined in \eqref{eq:model_main} is a classical model of sample covariance matrices in the random matrix theory, and its spectral properties have been intensively studied in the regime \eqref{eq:asymptotic} in the last several decades. 

At a global scale, the limiting spectral distribution (LSD) of the ESD $\mu^{S_N}=N^{-1}\sum_{k=1}^N\delta_{\lambda_k(S_N)}$ has been described in the groundbreaking paper by Mar\v cenko and Pastur \cite{marvcenko1967distribution}. In the important case where $S_N=\frac 1n Z_N Z_N^*$, sometimes referred to as the {\em white noise} model, the limiting spectral distribution of $\mu^{S_N}$ is known as Mar\v cenko-Pastur distribution and admits the following closed-form expression
$$
\mathbb{P}_{MP}(\dd\lambda):=\left(1-r^{-1}\right)_+\delta_0(\dd\lambda)+\frac{\sqrt{[(\lambda^+-\lambda)(\lambda-\lambda^-)]_+}}{2\pi r\lambda}\dd \lambda\ ,\qquad \lambda^{\pm} = \left(1\pm\sqrt{r}\right)^2\, , 
$$
where $x_+:=\max(x, 0)$. Later, this result was improved by many others, see for instance \cite{wachter1978strong,jonsson1982some,yin1986limiting,silverstein1995empirical,silverstein1995strong}. In \cite{silverstein1995strong}, Silverstein proved that for the model $S_N$ defined in \eqref{eq:model_main}, if $\mu^{\Gamma_N}$ weakly converges to a certain probability $\nu$ supported on $\nonnegR$ (not necessarily with compact support), then almost surely, the ESD $\mu^{S_N}$ weakly converges to a deterministic distribution $\mu$, whose Stieltjes transform $g_\mu$ is the unique solution with positive imaginary part of the equation
\begin{equation}
g_\mu(z)=\int\frac{1}{s(1-r-rz g_\mu(z))-z}\dd\nu(s)\in\C^+,\quad \forall z\in\C^+\, . \label{eq:intro:stieltjes}
\end{equation}
Central limit theorems have also been established for linear spectral statistics $\sum_{i=1}^N f(\lambda_i(S_N))$, see for instance \cite{jonsson1982some,johansson2000shape,bai2004clt,najim2016gaussian}.

At a local scale, the convergence and fluctuations of individual eigenvalues have been studied, with a special emphasis on the eigenvalues located near each edge of the connected components (bulk) of the LSD of $S_N$. The spiked eigenvalues, that is those which stay away from the bulk of the LSD, have also attracted a lot of attention.

For the white noise model, the support of Mar\v cenko-Pastur's LSD is $[(1-\sqrt{r})^2,(1+\sqrt{r})^2]$, with $\{0\}$ if $r>1$. Geman \cite{geman1980limit} showed that $\lmax(S_N)\to (1+\sqrt{r})^2$ almost surely under moment conditions on the entries. Later, Bai et al. \cite{yin1988limit,bai1988note,bai1993limit} showed that $\lmax(S_N)$ almost surely converges to a finite limit if and only if the fourth moment $\E|Z_{1,1}^{(1)}|^4$ of the entries is finite. Concerning the fluctuations of $\lambda_{\max}(S_N)$, they were first studied by Johansson \cite{johansson2000shape} for standard Gaussian complex entries and by Johnstone \cite{johnstone2001distribution} for standard Gaussian real entries. They both established that 
\begin{equation}
\gamma_N \,n^{2/3} \left (\lambda_{\max}(S_N) - (1+\sqrt{r_N})^2\right) \quad \textrm{where} \quad r_N=\frac Nn \quad \textrm{and} \quad \gamma_N= \frac{r_N^{1/6}}{(1+\sqrt{r_N})^{4/3}}
\label{eq:johnstone}
\end{equation}
converges in distribution to Tracy-Widom (TW) distributions as $N,n\to \infty$, introduced in \cite{tracy1993level, tracy1996orthogonal}, to describe the fluctuations of the largest eigenvalues of GUE and GOE random matrices. 

For general sample covariance matrices \eqref{eq:model_main}, the condition that the spectral norm of $\Gamma_N$ is uniformly bounded:
$$
\sup_{N\ge 1} \| \Gamma_N\| = \sup_{N\ge 1} \lambda_{\max}(\Gamma_N) \ <\ \infty 
$$
implies that the LSD $\mu$ (defined by its Stieltjes transform $g_\mu$ which satisfies \eqref{eq:intro:stieltjes}) has a bounded support. In this case, 
El Karoui  \cite{karoui2007tracy} and Lee and Schnelli \cite{lee2016tracy} established Tracy-Widom type fluctuations of the largest eigenvalue in the complex and real Gaussian case respectively. By establishing a local law, Bao et al \cite{bao2015universality}, and Knowles and Yin \cite{knowles2017anisotropic}
extended the fluctuations of the largest eigenvalue for general entries.

The case of spiked models has been addressed by Baik et al \cite{baik2005phase,baik2006eigenvalues} where some eigenvalues (the spikes) may separate from the bulk. 
In \cite{baik2005phase} where the so-called BBP phase transition phenomenon is described, Baik et al. study the case where $\Gamma_N$ has exactly 
$m$ non-unit eigenvalues $\ell_1\ge \cdots \ge \ell_m$. For complex Gaussian entries, they fully describe the fluctuations of $\lmax(S_N)$ for different configurations of the $\ell_i$'s. Assume for instance that $\ell_1$ is simple (cf. the original paper for the general conditions) then (a) if $\ell_1\le1+\sqrt{r}$, $\lmax(S_N)$ has asymptotically TW fluctuations at speed $n^{2/3}$; (b) if $\ell_1>1+\sqrt{r}$, the sequence
\begin{equation}\label{eq:spiked-finite}
\frac{\sqrt{n}}{\sqrt{\ell_1^2-\ell_1^2r_N/(\ell_1-1)^2}}\left(\lmax(S_N)-\left(\ell_1+\frac{\ell_1r_N}{\ell_1-1}\right)\right)
\end{equation}
is asymptotically Gaussian. In \cite{baik2006eigenvalues} Baik and Silverstein consider general entries and prove the strong convergence of the spiked eigenvalues; Bai and Yao \cite{bai2008central} consider the spiked model with supercritical spikes (corresponding to the case (b) above) and general entries and establish Gaussian-type fluctuations for the spiked eigenvalues. Other results are, non exhaustively, \cite{benaych2011fluctuations,benaych2012singular,couillet2013fluctuations, bai2012sample}. 

To make a rough conclusion from these results, $\lmax(S_N)$ does not in general approach the largest eigenvalue of $\lmax(\Gamma_N)$. Moreover, if $\lmax(S_N)$ converges to the bulk edge of the LSD of $\mu^{S_N}$, then it often has Tracy-Widom fluctuation at the scale $n^{2/3}$. If $\lmax(S_N)$ converges to a point outside the bulk, it often has Gaussian-type fluctuation at the scale $n^{1/2}$. 

The previously mentionned results are limited to the case where $\lmax(\Gamma_N)$ is uniformly bounded. There are however interesting cases where $\lmax(\Gamma_N)$ goes to infinity, see for instance Forni et al. \cite{forni2000generalized} in a context of econometrics.

Recently and mainly fostered by principal component analysis (PCA) in high dimension, there has been a renewed interest in the case where a small number of spiked eigenvalues of the population covariance matrix goes to infinity while the rest of the population eigenvalues remains bounded. Let us mention in growing generality Jung and Marron \cite{jung2009pca}, Shen et al. \cite{shen2013surprising}, Wang and Fan \cite{wang2017asymptotics}, Cai et al. \cite{cai2017limiting}. In the latter, a complete description of the various scenarios of the spikes and their multiplicity is considered, and the first non-spiked eigenvalue's fluctuations are established. In \cite{ledoit2018optimal}, Ledoit and Wolf consider a similar framework referred to as the "Arrow model".

In this article, we complement the general picture by considering population covariance matrices with unbounded limiting spectral distribution. 
Such a case arises in the context of long memory stationary processes and is not covered by the existing results. In the framework considered here, we are not in the case where a majority of 
the population eigenvalues remains bounded. In particular, the assumptions in \cite{wang2017asymptotics,cai2017limiting} fail to hold.

\subsection*{Description of the main results.} Let $S_N$ be defined in \eqref{eq:model_main} and assume that $(\mu^{\Gamma_N})$ is tight with 
$\lim_{N\to\infty} \lmax(\Gamma_N)= \infty$, then we establish in Proposition \ref{prop:main_asympt} that 
\begin{equation}
\frac{\lmax(S_N)}{\lmax(\Gamma_N)}\xrightarrow[N,n\to\infty]{} 1 \label{eq:intro:asympt}\ 
\end{equation}
in probability. This convergence is improved to an almost sure (a.s.) convergence if either the $Z_{i,j}^{(N)}$'s are standard (real or complex) gaussian, or stem from the top left corner of an infinite array $(Z_{i,j},i,j\in \N)$ of i.i.d. random variables. In the case of a triangular array, one might expect an a.s. convergence if $\lmax(S_N)$ concentrates sufficiently fast around its expectation.

In order to describe the fluctuations of $\lmax(S_N)$, we assume in addition  that $(\Gamma_N)$ satisfies the following {\em spectral gap condition}
\begin{equation}
\varlimsup_{N\to\infty} \frac{\lambda_2(\Gamma_N)}{\lambda_{\max}(\Gamma_N)}\ <\ 1\ , \label{eq:gap_condition}    
\end{equation}
where $\lambda_2(\Gamma_N)$ is the second largest eigenvalue of $\Gamma_N$, and that either the $Z^{(N)}_{i,j}$'s are standard Gaussian or the $\Gamma_N$'s have a {\em block-diagonal structure}
\begin{equation}
    \Gamma_N=\begin{pmatrix}\lmax(\Gamma_N) & 0 \\ 0 & \Gammasf_{N-1} \end{pmatrix}\, . \label{eq:struct_condition}
\end{equation}
In this case, the following fluctuation result, stated in Theorem \ref{th:main_fluct}, holds: 
\begin{equation}
    \sqrt{n}\left(\frac{\lmax(S_N)}{\lmax(\Gamma_N)}-1-\frac{1}{n}\sum_{k=2}^{N}\frac{\lambda_k(\Gamma_N)}{\lmax(\Gamma_N)-\lambda_k(\Gamma_N)}\right)
    \xrightarrow[N,n\to\infty]{\mathcal D}
    \gNr(0,\sigma^2) \label{eq:intr:fluct}
\end{equation}
where ``$\cvweak$" denotes the convergence in distribution, $\sigma^2=\E|Z_{1,1}|^4-1$  and $\gNr$ stands for the real Gaussian distribution.

Notice that in \eqref{eq:intr:fluct}, the term $\beta_N:=\frac{1}{n}\sum_{k=2}^{N}\frac{\lambda_k(\Gamma_N)}{\lmax(\Gamma_N)-\lambda_k(\Gamma_N)}$ goes to zero (see for instance Remark \ref{rem:beta-n-to-zero} below), however $\sqrt{n} \beta_N$ may not converge to zero as shown in Example \ref{example:beta-n}.

These results are then applied to long memory stationary processes.

\subsection*{Long memory stationary process.} A process $(\Xc_t)_{t\in\Z}$ is (second order) stationary if the following conditions are satisfied:
$$
\E|\Xc_t|^2<\infty\, ,\qquad \E \Xc_t=\E \Xc_0\qquad \textrm{and} \qquad  \cov(\Xc_{t+h},\Xc_t)=\cov(\Xc_h,\Xc_0)=\gamma(h) \qquad \forall t,h\in\Z$$
where $\cov(\Xc_{t+h},\Xc_t) =  \E (\Xc_{t+h}- \E\Xc_{t+h})\overline{(\Xc_t-\E \Xc_t)}$ and $\gamma:\mathbb{Z}\to \C$ is some positive definite function, usually called the autocovariance function of the process. Note that $\gamma(0)$ is positive and $\gamma(-h)=\overline{\gamma(h)}$ for all $h\in\Z$. By stationarity, the covariance matrices $T_N(\gamma)$ of the process 
\begin{equation}
T_N(\gamma) :=    \cov \begin{pmatrix} \Xc_{t+1} \\ \vdots \\ \Xc_{t+N}\end{pmatrix}=\begin{pmatrix} \gamma(0) & \gamma(-1) & \dots & \gamma(-N+1) \\ \gamma(1) & \ddots  & \ddots & \vdots \\ \vdots & \ddots & \ddots & \gamma(-1) \\ \gamma(N-1) & \dots & \gamma(1) & \gamma(0) \end{pmatrix} \label{eq:autocov}
\end{equation}
are positive semidefinite Hermitian Toeplitz matrices.

By Herglotz's Theorem, there exists a finite positive measure $\alpha$ on $(-\pi,\pi]$, the {\em symbol} of $T_N(\gamma)$, whose Fourier coefficients are exactly $\gamma(h)$, i.e.
$$
\gamma(h)=\frac{1}{2\pi}\int_{(-\pi,\pi]}e^{-\ii hx}\dd\alpha(x)\,, \quad \forall h\in\Z\, .
$$
Depending on the context, we may write $T_N(\gamma)$ or $T_N[\alpha]$.

Tyrtyshnikov and Zamarashkin generalized 
in \cite{tyrtyshnikov2002toeplitz} a result of Szeg{\H o} and proved that the following equality holds 
\begin{equation}
    \lim_{N\to\infty}\frac{1}{N}\sum_{k=1}^N \varphi(\lambda_k(T_N(\gamma)))=\frac{1}{2\pi}\int_{-\pi}^{\pi}\varphi(f_\alpha(x))\dd x\,, \label{eq:szego_lim}
\end{equation}
where $\varphi:\R \mapsto \R$ is continuous with compact support and $f_\alpha\in L^1(-\pi,\pi)$ is the density of the absolutely continuous part of $\alpha$ with respect to the Lebesgue measure $\dd x$ on $(-\pi,\pi]$, called the {\em spectral density} of $T_N(\gamma)$. The equality \eqref{eq:szego_lim} can be interpreted as the vague convergence of probability measures $\mu^{T_N(\gamma)}$ to the measure $\nu$ defined by the integral formula
\begin{equation}
\int\varphi\dd \nu=\frac{1}{2\pi}\int_{-\pi}^{\pi}\varphi(f_\alpha(x))\dd x \quad \forall \varphi\in C_b, \label{eq:def_nu}
\end{equation}
where $C_b$ denotes the space of all bounded continuous functions. The measure $\nu$ being a probability, the sequence $\mu^{T_N(\gamma)}$ is tight, and the vague convergence coincides with the weak convergence.

The process is usually said to have {\em short memory} or {\em short range dependence} if $\sum_{h\in\Z}|\gamma(h)|<\infty$. Otherwise, if $$
\sum_{h\in\Z}|\gamma(h)|=\infty\ ,
$$ the process $(\Xc_t)$ has  {\em long memory} or {\em long range dependence}\footnote{There are several definitions of long range dependance, all strongly related but not always equivalent, see for instance \cite[Chapter 2]{pipira2018long_range}.}. 

In this article we require that the autocovariance function $\gamma$ of a long memory stationary process satisfies
\begin{equation}
    \gamma(h)=\frac{L(h)}{(1+|h|)^{1-2d}}\, ,\quad \forall h\in\Z \label{eq:intro:long_memory}
\end{equation}
for some $d\in(0,1/2)$ and a function $L:\R\to\R$ slowly varying at $\infty$,  that is, a function satisfying $L(y)>0$ for $|y|$ large enough, and
$$
\lim_{y\to\infty}\frac{L(xy)}{L(y)}= 1\quad \forall x>0\, .
$$
In this case $\gamma$ is real and even and $L$ is an even function as well. Matrix $T_N(\gamma)$ is real symmetric and $(\Xc_t)$ is a long memory process. In addition, $\lambda_{\max}(\Gamma_N)\xrightarrow[N\to \infty]{} \infty$, see for instance Theorem \ref{th:main_toepl}.

\subsubsection*{The largest eigenvalue associated with a long memory stationary Gaussian processes.} 
Given a centered stationary process $(\Xc_t)_{t\in\Z}$ with autocovariance function defined by \eqref{eq:intro:long_memory}, one can study the spectral properties of the sample covariance matrix
\begin{equation}
    Q_N:=\frac{1}{n}X_NX_N^*=\frac{1}{n}\sum_{j=1}^n X_{\cdot,j}X_{\cdot,j}^* \label{eq:model_proc}
\end{equation} 
where $X_N$ is a $N\times n$ random matrix whose columns $(X_{\cdot,j},1\le j\le n)$ are i.i.d copies of the random vector $\Xc_{1:N}=(\Xc_1,\dots,\Xc_N)^\tran$. 

Let $T_N(\gamma)$ be the covariance matrix of $(\Xc_t)$, it has been recalled that $\mu^{T_N(\gamma)}$ weakly converges.
Since the process is Gaussian, $Q_N$ can be written in the form of $S_N$ in \eqref{eq:model_main} with $\Gamma_N=T_N(\gamma)$ and the ESD $\mu^{Q_N}$ weakly converges with probability one to a deterministic probability measure $\mu$ by \cite[Theorem~1.1]{silverstein1995strong}.

In order to study the behavior of $\lmax(Q_N)$ and to apply the results already presented, note that the process being gaussian, the matrix model \eqref{eq:model_main} has the same spectral properties as a model where $\Gamma_N$ is replaced by the diagonal matrix obtained with $\Gamma_N$'s eigenvalues. In particular, the block-diagonal structure condition \eqref{eq:struct_condition} is automatically satisfied. It remains to verify
the spectral gap condition \eqref{eq:gap_condition} and that $\lmax(T_N(\gamma))$ goes to infinity. In Theorem \ref{th:main_toepl}, we describe the asymptotic behaviour of the $k$th largest eigenvalue $\lambda_k(T_N(\gamma))$ for any fixed $k$, and prove that there exist positive numbers $a_1>a_2\ge a_3\ge \dots >0$ such that for any $k\ge 1$,  
$$
\lambda_k(T_N(\gamma))\sim a_kN\gamma(N)\qquad \textrm{and}\qquad \lim_{N\to\infty}\frac{\lambda_2(T_N(\gamma))}{\lmax(T_N(\gamma))}=\frac{a_2}{a_1}<1\, , 
$$
hence the spectral gap condition \eqref{eq:gap_condition} holds. Moreover, standard properties of slowly varying functions \cite[Prop. 1.3.6(v)]{bingham1989regular} yield 
that $N\gamma(N)\to\infty$ hence $\lambda_k(T_N(\gamma))\to\infty$ and in particular $\lmax(T_N(\gamma))\to\infty$. As a corollary, we obtain the asymptotics and fluctuations of the largest eigenvalue $\lmax(Q_N)$ for Gaussian long memory stationary processes with autocovariance function defined in \eqref{eq:intro:long_memory}.

We now point out two references of interest: In the (non-Gaussian) case where the symbol $\alpha$ is absolutely continuous with respect to the Lebesgue measure and under additional regularity conditions on $(\Xc_t)$, Merlevède and Peligrad \cite{merlevede2016empirical} have established the convergence of the ESD $\mu^{Q_N}$ toward a certain deterministic probability distribution. In a context of a stationary Gaussian field, Chakrabarty et al. \cite{chakrabarty2016random} studied large random matrices associated with long range dependent processes.

\subsection*{Organization.} 

Our paper is organized as follows. In Section \ref{sec:main_results} we state the assumptions and main results of the article: Proposition~\ref{prop:main_asympt} and Theorem~\ref{th:main_fluct} are devoted to the limiting behaviour and fluctuations of $\lmax(S_N)$; the spectral gap condition for a Toeplitz matrix $\Gamma_N$ is studied in Theorem~\ref{th:main_toepl}; finally Corollary~\ref{th:process} builds upon the previous results and describes the behaviour and fluctuations of covariance matrices based on samples of stationary long memory Gaussian processes. In Section \ref{sec:miscellany}, we provide examples, numerical simulations and mention some open questions. Section \ref{sec:proof_main} and Section \ref{sec:proof_toepl} are dedicated to the proofs of the main theorems.

\subsection*{Acknowledgement.}The authors would like to thank Walid Hachem for useful discussions and the two referees for helpful comments which improved the presentation of the paper. 

\section{Notations and main theorems}
\label{sec:main_results}
\subsection{Notations and assumptions}
\subsubsection*{Notations.} Given $x\in\R$, denote by $\lfloor x\rfloor$ the integer satisfying $ \lfloor x\rfloor\le x <\lfloor x\rfloor+1$. 
For vectors $u,v$ in $\R^N$ or $\C^N$, $\langle u,v\rangle = \sum_{i=1}^N u_i\bar{v}_i$ denotes the scalar product and $\|u\|$ the Euclidean norm of $u$.

For a matrix or a vector $A$, we use $A^\tran$ to denote the transposition of $A$, and $A^*$ the conjugate transposition of $A$; if $A$ is a $N\times N$ square matrix with real eigenvalues, we use $\lambda_1(A)\ge \dots\ge\lambda_N(A)$ to denote its eigenvalues, and sometimes denote $\lambda_1(A)=\lmax(A)$. The ESD $\mu^A$ 
of $A$ is defined as 
$$
\mu^A:=\frac 1N \sum_{k=1}^N \delta_{\lambda_k(A)}\, , 
$$
where $\delta_\lambda$ is the Dirac measure at $\lambda$. For a $N\times n$ matrix $M$ and integers $a,b\in \{ 1,\cdots, N\}$ and $c,d\in \{ 1,\cdots, n\}$, 
the following notations are used to deal with submatrices of $M$:
\begin{equation}
M_{a:b,\cdot}\, =\, (M_{i,j})_{a\le i\le b,1\le j\le n}\, ,\quad M_{\cdot,c:d}\,=\,(M_{i,j})_{1\le i\le N,c\le j\le d}\, ,\quad
M_{a:b,c:d}\, =\,(M_{i,j})_{a\le i\le b,c\le j\le d}\, .\label{eq:def_submatrix}
\end{equation}
By convention, these subscripts have higher priority than the transposition or conjugate transposition, for example $M_{a:b,c:d}^*:=(M_{a:b,c:d})^*$ is the conjugated transposition of the submatrix $M_{a:b,c:d}$. For a matrix $A$, we write its operator norm as $\|A\|=\sup_{\|v\|=1}\|Av\|$ and its Frobenius norm $
\|A\|_F=\sqrt{\sum\nolimits_{i,j}|A_{i,j}|^2}$.

If $c=(c_k)_{k\in\Z}$ is a sequence of complex numbers, the $N\times N$ Toeplitz matrix $(c(i-j))$ is denoted by $T_N(c)$. If moreover the sequence $(c_k)$ is a positive-definite function $c: \mathbb{Z}\to \C$ and admits by Herglotz's theorem the representation 
$$
c_k = \frac 1{2\pi} \int_{(-\pi,\pi]} e^{-\ii kx} \dd \alpha(x)\, , 
$$  
then $\alpha$ is called the symbol of $T_N(c)$ which will sometimes be written $T_N[\alpha]$. If moreover $\alpha$ admits a density with respect to Lebesgue's measure, i.e. $\dd \alpha(x) = f(x) \dd x$, $T_N(c)$ will occasionnally be denoted by $T_N[f]$. Notice that if $T_N(c)$ is the covariance matrix of a stationary process as in \eqref{eq:autocov} then $f(x)$ (if it exists) is called the spectral density of the process.

Given two complex sequences $x_n, y_n$ we denote 
\begin{equation}
x_n\sim y_n\ \Leftrightarrow\  \lim_{n\to\infty}\left(\frac{x_n}{y_n}\right)=1\qquad \textrm{and}\qquad 
x_n\doteq y_n\ \Leftrightarrow\ \lim_{n\to\infty}(x_n-y_n)=0\, . 
\end{equation}
The notations $x_n=o(1)$ and $x_n=O(1)$ respectively mean $\lim_{n\to\infty} x_n=0$ and $\varlimsup_{n\to\infty} |x_n|<\infty$. These notations are also applicable to functions with continuous arguments. If $X_n,X$ are random variables, the notations $X_n=o(1)$ and $X_n=o_P(1)$ respectively mean that $\lim_{n\to\infty} X_n=0$ almost surely and in probability. The notations $X_n\cvweak X$ and $X_n\cvprob X$ respectively denote convergence in distribution and in probability. If $\mu,\mu_n$ are measures, we denote with a slight abuse of notation $\mu_n\cvweak \mu$ for the weak convergence of $\mu_n$ to $\mu$. 

Given a random variable $Y$ or a sub-algebra $\mathcal G$, we denote by $\E_Y(X)$ and $\E_{\mathcal G}(X)$ the conditional expectation of the random variable $X$ with respect to $Y$ and to ${\mathcal G}$.

We denote by ${\mathcal N}_\R(m,\sigma^2)$ the real Gaussian distribution with mean $m$ and variance $\sigma^2$; we refer to ${\mathcal N}_\R(0,1)$ as the standard real Gaussian distribution. A complex random variable $Z$ is distributed according to the standard complex Gaussian distribution if $Z=U+\ii V$ where $U,V$ are independent, each with distribution ${\mathcal N}_\R(0,1/2)$. In this case we denote $Z\sim {\mathcal N}_\C(0,1)$. For a symmetric semidefinite positive matrix $T$, denote by ${\mathcal N}_\R(0,T)$ the distribution of a centered Gaussian vector with covariance matrix $T$.

In the proofs we use $C$ to denote a constant that may take different values from one place to another. 

\subsubsection*{Assumptions.}We state our results under one or several of the following assumptions:
\vspace{.3cm}
\begin{enumerate}[leftmargin=*, label={\bf A\arabic{enumi}}]
    \item \label{ass:model_setting}(Model setting) Let $S_N$ be $N\times N$ random matrices defined as
    $$S_N=\frac{1}{n}\Gamma_N^\frac{1}{2}Z_NZ_N^*\Gamma_N^\frac{1}{2}$$
    where $Z_N=\left(Z_{i,j}^{(N)}\right)_{1\le i\le N, 1\le j\le n}$ are $N\times n$ matrix whose entries $Z_{i,j}^{(N)}$ are real or complex. The $Z_{i,j}^{(N)}$'s are i.d. random variables for all $i,j, N$, and independent across $i,j$ for each $N$, satisfying
    $$\E Z_{i,j}^{(N)}=0, \quad \E|Z_{i,j}^{(N)}|^2=1 \quad\text{ and } \quad\E|Z_{i,j}^{(N)}|^4<\infty,$$
    and $\Gamma_N$ are $N\times N$ positive semidefinite Hermitian deterministic matrices.
    \end{enumerate}
    Notice that in the assumption above, we do not require that $\E\left( Z_{i,j}^{(N)}\right)^2=0$ in the case of complex r.v.  but this is automatically fulfilled in the case of standard complex gaussian random variables.
 \begin{enumerate}[leftmargin=*, label={\bf A\arabic{enumi}}]
\setcounter{enumi}{1}   
    \item \label{ass:asympt_spec_st}(Asymptotic spectral structure of $\Gamma_N$) 
    Given a sequence of $N\times N$ positive semidefinite deterministic matrices $\Gamma_N$, the empirical spectral distribution
    $$\mu^{\Gamma_N}:=\frac{1}{N}\sum_{k=1}^N\delta_{\lambda_k(\Gamma_N)}$$
    forms a tight sequence on $\nonnegR$, and the largest eigenvalue $\lmax(\Gamma_N)$ tends to $\infty$ as $N\to\infty$.
    \end{enumerate}

\begin{example} 
If $\mu^{\Gamma_N}\cvweak \nu$ with $\nu$ a non-compactly supported probability on $\nonnegR$, then \ref{ass:asympt_spec_st} holds.
\end{example}

\begin{example} Consider $\Gamma_N=\diag(\ell_1^{(n)},\cdots, \ell_m^{(n)},1,\cdots, 1)$ where $m=m(n)$ is such that $\frac{m(n)}n\to 0$ and where $\ell_i^{(n)} \nearrow \infty$ ($1\le i\le  m$),  then \ref{ass:asympt_spec_st} holds. The illustrative and simpler case where $\Gamma_N=\diag (\ell_1^{(n)},1,\cdots, 1)$ will be used hereafter.  
\end{example}
\begin{enumerate}[leftmargin=*, label={\bf A\arabic{enumi}}]
\setcounter{enumi}{2}
    \item \label{ass:infarray}(Subarray assumption on $Z_N$) For each $N$, $Z_N=Z_{1:N,1:n}$ is the top-left submatrix of an infinite matrix $Z=(Z_{i,j})_{i,j\ge 1}$, with $Z_{i,j}$ i.i.d random variables satisfying 
    $$\E Z_{i,j}=0, \quad \E|Z_{i,j}|^2=1 \quad\text{ and } \quad\E|Z_{i,j}|^4<\infty.$$
    \end{enumerate}
\begin{Rq} 
Compared to \ref{ass:model_setting}, this subarray assumption (where the matrix entries do not depend on $N,n$) is mainly needed to state almost sure convergence results (see Proposition \ref{prop:main_asympt} below) and to fully exploit Bai and Silverstein's results on spectrum confinement \cite{bai1998no}. Notice that under this assumption, interlacing properties of eigenvalues hold true since the realization of the random variables remains the same as $N,n\to \infty$. 
\end{Rq}
    \begin{enumerate}[leftmargin=*, label={\bf A\arabic{enumi}}]
\setcounter{enumi}{3}
    \item \label{ass:separation}(Spectral gap condition on $\Gamma_N$) The two largest eigenvalues $\lmax(\Gamma_N)$ and $\lambda_2(\Gamma_N)$ satisfy
     $$\varlimsup_{N\to\infty}\frac{\lambda_2(\Gamma_N)}{\lmax(\Gamma_N)}<1.$$ 
     \end{enumerate}
Notice that this spectral gap condition already appears in \cite{shen2013surprising,wang2017asymptotics,cai2017limiting}.
     
\begin{enumerate}[leftmargin=*, label={\bf A\arabic{enumi}}]
\setcounter{enumi}{4}     
     
    \item  \label{ass:diagonal}(Block-diagonal structure of $\Gamma_N$) For all $N$, $\Gamma_N$ has the block-diagonal form
    $$
    \Gamma_N=\begin{pmatrix}\lmax(\Gamma_N) & 0 \\ 0 & \Gammasf_{N-1} \end{pmatrix}\, , 
    $$
    where $\Gammasf_{N-1}$ is a $(N-1)\times (N-1)$ semidefinite positive Hermitian matrix.
\end{enumerate}
\vspace{.3cm}

\subsection{Main results}
We now present the main results of this article. Recall that the asymptotic regime $N,n\to\infty$ (cf. \eqref{eq:asymptotic}) stands for 
$$
N,n\to\infty\qquad \textrm{and}\qquad \frac Nn \to r\in (0,\infty)\, .
$$ Proposition \ref{prop:main_asympt} and Theorem \ref{th:main_fluct} describe the limiting behaviour and fluctuations of $\lmax(S_N)$ under generic assumptions. Theorem \ref{th:main_toepl} and Corollary \ref{th:process} specialize the previous results to Toeplitz covariance matrices and Gaussian long memory stationary processes.

\begin{prop} \label{prop:main_asympt}
Let $S_N$ be a $N\times N$ matrix given by \eqref{eq:model_main} and assume that \ref{ass:model_setting} and \ref{ass:asympt_spec_st} hold. Then 
$$
\frac{\lmax(S_N)}{\lmax(\Gamma_N)}\xrightarrow[N,n\to\infty]{\mathcal P} 1\,.
$$
If moreover either the random variables $Z_{ij}^{(N)}$ are standard (real or complex) Gaussian or Assumption \ref{ass:infarray} holds, then the above convergence holds almost surely.
\end{prop}
{This result already appears under different assumptions in \cite[Prop. 7.3]{ledoit2018optimal}, \cite[Th. 2.1]{cai2017limiting}.}
\begin{Rq} [consistency with the bounded case $\|\Gamma_N\|<\infty$]
Consider the simple case where $\Gamma_N=\diag(\ell_1,1,\cdots,1)$, where $\ell_1>1+ \sqrt{r}$ is fixed - see for instance \eqref{eq:spiked-finite}. Then it is well known (cf. \cite{baik2006eigenvalues}) that 
$$
\lmax(S_N) \quad \xrightarrow[N,n\to\infty]{\mathcal P}\quad  \ell_1+\frac{r\, \ell_1}{\ell_1-1}\, .
$$
Notice in particular that $\frac{\lmax(S_N)}{\ell_1} = 1+\frac{r}{\ell_1-1} + o_P(1)
$, which is heuristically consistent with Proposition \ref{prop:main_asympt} if one lets $\ell_1$ go to infinity.
\end{Rq}
\begin{Th}\label{th:main_fluct}
Let $S_N$ be a $N\times N$ matrix given by \eqref{eq:model_main} and assume that \ref{ass:model_setting}, \ref{ass:asympt_spec_st}  and \ref{ass:separation} hold.
Assume moreover that one of the following conditions is satisfied:
\begin{enumerate}[label=(\roman*), itemsep=-3pt]
\item Assumption \ref{ass:diagonal} holds,
\item The random variables $Z_{ij}^{(N)}$ are standard complex Gaussian,
\item The random variables $Z_{ij}^{(N)}$ are standard real Gaussian and matrices $\Gamma_N$ are real symmetric. 
\end{enumerate}
Consider the quantities
\begin{equation}
\beta_N := \frac{1}{n}\sum_{k=2}^{N}\frac{\lambda_k(\Gamma_N)}{\lmax(\Gamma_N)-\lambda_k(\Gamma_N)}\qquad \textrm{and}\qquad F_N:=\sqrt{n}\left(\frac{\lmax(S_N)}{\lmax(\Gamma_N)}-1-\beta_N\right). \label{eq:def_FN}
\end{equation}
Then 
\begin{equation}
    F_N\xrightarrow[N,n\to\infty]{\mathcal D} \gNr(0,\sigma^2)\, , \label{eq:fluct}
\end{equation}
where $\sigma^2=\E|Z_{1,1}^{(1)}|^4-1$.
\end{Th}
{Counterparts of Theorem \ref{th:main_fluct} appear under the assumption that the $(\lambda_i(\Gamma_N))$'s are bounded for $i\ge K$ and $K=o(N)$, see \cite[Th. 3.1]{wang2017asymptotics}, \cite[Th. 2.2]{cai2017limiting}. In this latter case, the quantity $\beta_N$ above can be replaced by $n^{-1} \sum_{K+1}^N \lambda_i(\Gamma_N)/(\lmax(\Gamma_N) - \lambda_i(\Gamma_N))$. Beware however that under our assumption, the full summation is required because there is no natural threshold $K$ if one does not assume boundedness on the majority of the population eigenvalues.}
\begin{Rq} Notice that if $\E|Z_{1,1}^{(1)}|^4=1$ then $\sigma^2=0$ in the previous theorem, hence 
$
F_N\xrightarrow[N,n\to\infty]{\mathcal P} 0\, .
$
Simulation 3 in Section \ref{sec:simulations} (see also Fig. \ref{fig:fluct_bin_nondiag}) supports this fact.
\end{Rq}

\begin{Rq}\label{rem:beta-n-to-zero} Under \ref{ass:asympt_spec_st} and \ref{ass:separation}, we have
$
\beta_N
\xrightarrow[N,n\to\infty]{} 0
$.
Indeed, by the spectral gap condition \ref{ass:separation} and the fact that $N=O(n)$ 
$$
\beta_N\ =\ \frac{1}{n}\sum_{k=2}^{N}\frac{\lambda_k(\Gamma_N)/\lmax(\Gamma_N)}{1-\lambda_k(\Gamma_N)/\lmax(\Gamma_N)}\ \le\ \frac{C}{N}\sum_{k=2}^{N}\frac{\lambda_k(\Gamma_N)}{\lmax(\Gamma_N)}\,.
$$
Since $\mu^{\Gamma_N}$ is tight, for any $\varepsilon\in (0,1)$ there exists $M>0$ s.t. $|\{k,\,:\lambda_k(\Gamma_N)>M\}|/N<\varepsilon$ where $|\{\cdot\}|$ denotes the cardinality of a set. Hence
$$
\varlimsup_{N,n\to\infty} \beta_N\,\le\, C\varlimsup_N\frac{M}{\lmax(\Gamma_N)}+C\varepsilon\,=\,C\varepsilon\,,$$
where we use the fact that $\lmax(\Gamma_N)\to\infty$ as $N\to\infty$ for the last equality. Notice however that $\sqrt{n}\beta_N$ may not go to zero as $N,n\to\infty$.
\end{Rq}

\begin{Rq}[consistency with the bounded case $\|\Gamma_N\|<\infty$, continued]
Consider again the case where $\Gamma_N=\diag(\ell_1,1,\cdots, 1)$ with $\ell_1>1+\sqrt{r}$ then
$$
\beta_N=\frac {N-1}n \frac 1{\ell_1- 1} = \frac{r_N}{\ell_1- 1} + O\left(\frac 1n \right)\, .
$$
In the case where $\ell_1\to\infty$, $F_N$ in \eqref{eq:def_FN} writes
$$
F_N=\sqrt{n}\left( \frac{\lmax(S_N)}{\ell_1} - \left( 1+\frac{r_N}{\ell_1-1}\right)
\right)+ O\left(\frac 1{\sqrt{n}}\right)
$$
and has Gaussian fluctuations. This formula is consistent with \eqref{eq:spiked-finite} which can be rewritten
$$
\frac{\sqrt{n}}{\sqrt{1+O(\ell_1^{-2})}} \left( \frac {\lmax(S_N)}{\ell_1} - \left( 1 +\frac{r_N}{\ell_1-1}\right)\right)\, .
$$
\end{Rq}
\begin{example}[various behaviours of $\sqrt{n} \beta_N$]
\label{example:beta-n} Consider $\Gamma_N=\diag(\ell_1^{(n)},1,\cdots, 1)$ where $\ell_1^{(n)} \nearrow \infty$, then
$$
\sqrt{n} \beta_N = \sqrt{n}\times  \left(\frac {N-1}n\right) \times \left(\frac{1}{\ell_1 - 1}\right) \quad \xrightarrow[N,n\to\infty]{} \quad \left\{
\begin{array}{lcl}
\infty&\textrm{if}& \ell_1^{(n)} \ll \sqrt{n}\, ,\\
r/a&\textrm{if}&  \ell_1^{(n)} = a\sqrt{n}\, , \\
0&\textrm{if}& \ell_1^{(n)} \gg \sqrt{n}\, .
\end{array}
\right.
$$ 
\end{example}

\subsection{Application to large sample covariance matrices associated with long memory processes} 
In order to apply the above results to Gaussian stationary processes with long memory, we need to verify the spectral gap condition of their autocovariance matrices. We first recall some definitions.

\begin{defi}[Regularly/Slowly varying functions] \label{defi:regular_slow}
A measurable function $R:\R\to \R$ is regularly varying at infinity if $R(y)>0$ for $|y|$ large enough and if there exists a real number $\rho$ s.t. for any $x>0$,
$$\lim_{y\to\infty}\frac{R(xy)}{R(y)}=x^\rho.$$
The number $\rho$ is called the index of the regular variation. If $\rho=0$, then we say that the function (often denoted by $L$ in this case) is slowly varying.

A sequence of real numbers $(c_k)_{k\in\Z}$ is regularly (resp. slowly) varying if $y\mapsto c_{\lfloor y\rfloor}$ is a regularly (resp. slowly) varying function.
\end{defi}
With Definition~\ref{defi:regular_slow}, long memory (long range dependence) stationary processes can be defined as follows.
\begin{defi} \label{defi:long_mem}
A stationary process $(\Xc_t)_{t\in\Z}$ has long memory or long range dependence if its autocovariance function $\gamma$ is regularly varying with index $\rho\in(-1,0)$.
\end{defi}
\begin{Rq} Notice that this definition is compatible with the definition of the autocovariance function provided in \eqref{eq:intro:long_memory}. 
In fact, assume that $\gamma(h)$ is given by \eqref{eq:intro:long_memory} then it is regularly varying with index $\rho=2d-1\in (-1,0)$.
Conversely, assume that $\gamma(h)$ is an even regularly varying sequence with $\rho\in (-1,0)$. Set $d =\frac{\rho+1}2$, then 
$
L(y)= \gamma(\lfloor y\rfloor ) (1+|y|)^{1-2d}
$ is a slowly varying function with $d\in (0,1/2)$ and 
$$
\gamma(h) = \frac{L(h)}{(1+|h|)^{1-2d}}\ .
$$
\end{Rq}
\begin{Rq}
Notice that definitions \ref{defi:regular_slow} and \ref{defi:long_mem} enable to consider complex processes $(\Xc_t)$, however the associated autocovariance function is necessarily real and cannot be complex.
\end{Rq}
\begin{Rq}
The above definition coincides with Condition II in \cite{pipira2018long_range} where the autocovariance function $\gamma$ satisfies \eqref{eq:intro:long_memory}. 
\end{Rq}

In this context, the spectral gap condition on autocovariance matrices of a long memory stationary process is ensured by the following theorem:
\begin{Th}\label{th:main_toepl} Suppose that $c=(c_h)_{h\in\Z}$ is an even ($c_h=c_{-h}$ for all $h\in\Z$) regularly varying sequence of index $\rho\in(-1,0)$, then there exist positive numbers $a_1^{(\rho)}>a_2^{(\rho)}\ge a_3^{(\rho)} \ge \dots >0$ such that for any fixed $k\ge 1$,
$$\lim_{N\to\infty}\frac{\lambda_k(T_N(c))}{Nc_N}=a_k^{(\rho)}.$$
In particular, 
$$
\lambda_k(T_N(c))\sim a_k^{(\rho)} Nc_N\to \infty\qquad \textrm{and}\qquad  \lim_{N\to\infty}\frac{\lambda_2(T_N(c))}{\lmax(T_N(c))}=\frac{a_2^{(\rho)}}{a_1^{(\rho)}}\, <\, 1\, .
$$
\end{Th}

\begin{Rq} \label{rem:non-trivial} The fact that $\lambda_k(T_N(c))\to\infty$ for any fixed $k\ge 1$ underlines the difference between the model under study and the models studied in \cite{bai2008central}, \cite{bai2012sample} and \cite{cai2017limiting} where it is always assumed that there is a finite number of eigenvalues of the population covariance matrix that are spiked \cite{bai2008central,bai2012sample} or go to infinity \cite{cai2017limiting}. 
\end{Rq}

\begin{defi} \label{defi:circ_sym_gaus_proc} We say that a complex Gaussian process $(\Xc_t)_{t\in\Z}$ is circularly symmetric if for any $N\in\N$ the Gaussian vector $\Xc_{1:N}:=(\Xc_1,\dots,\Xc_N)^\tran$ is circularly symmetric, i.e. for any $\phi\in\R$, the vector $e^{i\phi}\Xc_{1:N}$ has the same distribution as $\Xc_{1:N}$.
\end{defi}

\begin{Rq}
Notice in particular that such a process is centered and satisfies that $\E \Xc_{1:N}\Xc_{1:N}^\tran=0$, for all $N\ge 1$.
\end{Rq}
As a canonical example, a standard complex Gaussian vector $X=(X_1,\cdots, X_N)^\tran$ where the $X_i$'s are i.i.d. and 
${\mathcal N}_\C(0,1)$-distributed is circularly symmetric with 
$$
\E X X^* = I_N\quad \textrm{and}\quad \E X X^\tran =0\, .
$$
For such a vector $X$, we will denote $X\sim{\mathcal N}_{\C}(0,I_N)$.

As a by-product of the above theorems, we have the following result on the largest eigenvalue of sample covariance matrices of a Gaussian long memory stationary
process (recall that such a process admits a real autocovariance function).

\begin{corol}\label{th:process}Suppose that $(\Xc_t)_{t\in\Z}$ is a real centered (resp. complex circularly symmetric) Gaussian stationary process with long range dependence in the sense of definition \ref{defi:long_mem}. Let
$$Q_N=\frac{1}{n}X_NX_N^*$$
where $X_N$ are $N\times n$ random matrices whose columns are i.i.d copies of the random vector $\Xc_{1:N}=(\Xc_1,\dots,\Xc_N)^\tran$. Then 
\begin{equation}\frac{\lmax(Q_N)}{\lmax(T_N)}\, \xrightarrow[N,n\to\infty]{a.s.}\,  1  \label{eq:proc_assym}\end{equation}
and 
\begin{equation}\sqrt{n}\left(\frac{\lmax(Q_N)}{\lmax(T_N)}-1-\frac{1}{n}\sum_{k=2}^{N}\frac{\lambda_k(T_N)}{\lmax(T_N)-\lambda_k(T_N)}\right)
\,\xrightarrow[N,n\to\infty]{\mathcal D}\, \gNr\left(0,\sigma^2\right), \label{eq:proc_fluct}\end{equation}
where $T_N=T_N(\gamma)$ is the autocovariance matrix of the process defined in \eqref{eq:autocov} and where 
$$
\sigma^2~=~\frac{\mathbb{E} |\Xc^4_{1}|}{\gamma^2(0)} ~-~1= \left\{ 
\begin{array}{ll}
2&\textrm{if}\ \Xc_1\ \textrm{is real}\\
1&\textrm{if}\ \Xc_1\ \textrm{is complex}\\
\end{array}
\right.\, .
$$
\end{corol}

Corollary~\ref{th:process} being an easy consequence of Proposition~\ref{prop:main_asympt} and Theorems \ref{th:main_fluct} and \ref{th:main_toepl}, we provide its proof hereafter.

\begin{proof} Let $X$ be a centered $N$-dimensional random vector either real or circularly symmetric complex gaussian with (real) covariance matrix $T$. Then $X$ writes $X=T^{1/2} Z$ where $Z\sim {\mathcal N}_{\R}(0,I_N)$ or ${\mathcal N}_{\C}(0,I_N)$ depending on whether ${X}$ is real or complex. In fact, if $T$ is invertible 
then $Z=T^{-1/2}X$ has the required properties. 

If not, $T=O\diag(d_1^2,\cdots,d_p^2,0\cdots) O^\tran$ with $O$ orthogonal and $d_i>0$. Let $Y=(0,\dots,0,Y_{p+1},\dots,Y_N)^\tran$ with $Y_k$  i.i.d. standard Gaussian random variables, either real or complex (depending on $X$), and independent from $X$. Let 
$$
Z=O\diag(d_1^{-1},\cdots, d_p^{-1},0\cdots)O^\tran X + O\, Y
$$
then $\cov(Z)=I_N$ and if $X$ is complex, then $\E Z Z^\tran=0$. In particular, $Z$ is a standard gaussian random vector and a covariance computation yields
$$
\cov(X-T^{1/2}Z)=0\, ,
$$
which implies that $X=T^{1/2} Z$ almost surely.

Then for any $N$ and almost surely, the representation $Q_N=\frac 1n T_N^{1/2} Z_N Z_N^* T_N^{1/2}$ holds with $(T_N)$ the autocovariance matrices of process $(\Xc_t)_{t\in\Z}$. In Section \ref{sec:intr} we have noticed that $\mu^{T_N}$ converges weakly, and by Theorem~\ref{th:main_toepl}, $\lmax(T_N)\to\infty$ and $(T_N)$ satisfy the spectral gap condition. By Proposition \ref{prop:main_asympt} and Theorem \ref{th:main_fluct}, the results follow. 

\end{proof}

%

\section{Additional applications and simulations}
\label{sec:miscellany}
\subsection{Additional applications to Gaussian stationary processes}
\label{sec:further_app}
Although Definition~\ref{defi:long_mem} is a common definition of long memory, it can seem restrictive as it requires that the autocovariance function has at most a finite  number of nonpositive values.
In the first example hereafter we consider Gaussian processes with autocovariance functions either complex or with alternate signs. In the two subsequent examples, we relate our results with other definitions of long memory, via linear representation or via the autocovariance density.


\subsubsection*{Covariance matrices with alternating signs of entries.} 
Let $(\Xc_t)_{t\in\Z}$ be a centered Gaussian stationary process with autocovariance
function $\gamma^\Xc$. Let $\theta\in (-\pi,\pi],\, \theta \neq 0$ be fixed and consider the process $(\Yc_t=e^{\ii t\theta}\Xc_t)_{t\in\Z}$. This process is a Gaussian stationary process with autocovariance function $\gamma^{\Yc}(t)=e^{\ii t\theta} \gamma^{\Xc}(t)$ and 
$$
\Yc_{1:N} = \Sigma^{\theta} \Xc_{1:N}
$$
where $\Sigma^{\theta}=\diag(e^{\ii k\theta}, 1\le k\le N)$ is a unitary matrix. Notice that if $(\Xc_t)$ is complex circularly symmetric, then so is $(\Yc_t)$ but if $(\Xc_t)$ is real 
then $(\Yc_t)$ is either complex Gaussian but not circularly symmetric if $\theta\neq\pi$ or real Gaussian with alternate signs if $\theta=\pi$. 

Let $X^{\Xc}_N$ (resp. $X^{\Yc}_N$) a $N\times n$ matrix whose columns are i.i.d. copies of the vector $\Xc_{1:N}$
(resp. $\Yc_{1:N}$) and assume that the process $(\Xc_t)$ fulfills the assumptions of Corollary \ref{th:process}. Then 
$$
Q_N^{\Yc}=\frac 1n X_n^{\Yc} (X_n^{\Yc})^*=\frac 1n \Sigma^{\theta} X_n^{\Xc} (X_n^{\Xc})^* (\Sigma^{\theta})^*
= \Sigma^{\theta} Q_N^{\Xc} (\Sigma^{\theta})^*\qquad \textrm{where} \qquad Q_N^{\Xc}=\frac 1n X_n^{\Xc} (X_n^{\Xc})^*\, .
$$
In particular, $\lmax(Q^{\Yc}_N)=\lmax(Q^{\Xc}_N)$ satisfies \eqref{eq:proc_assym} and \eqref{eq:proc_fluct}.  
In this example the positivity constraint of the autocovariance function is relaxed.

%

\subsubsection*{Linear processes with long memory.}
If the process $(\Xc_t)_{t\in\Z}$ has a linear representation
$$\Xc_t=\sum_{j=0}^\infty \psi_j\epsilon_{t-j}$$
where $(\epsilon_t)_{t\in\Z}$ is a sequence of i.i.d real valued standard Gaussian r.v.'s and $\psi_j\sim j^{d-1}L(j)$ as $j\to\infty$ with $d\in(0,1/2)$ and $L$ a slowly varying function at $\infty$, then it is well known (c.f. for example \cite[Corollary~2.2.10]{pipira2018long_range}) that its autocovariance function $\gamma$ is regularly varying with index $\rho=2d-1$, and more precisely we have 
$$\gamma(h)\sim h^{2d-1}L^2(h)B(1-2d,d)$$
where $B(1-2d,d)=\int_0^1 x^{-2d}(1-x)^{d-1}\dd x$ is the beta-function. Corollary \ref{th:process} can be applied in this case.

\subsubsection*{Long range dependence defined through spectral density}
Among the various definitions of long range dependence, there is an important one which defines the long range dependence through the spectral density, that is, if the symbol of the autocovariance matrices $T_N$ has a density $f:(-\pi,\pi]\rightarrow \R^+$ satisfying
\begin{equation}
f(x)=|x|^{-2d}L\left (\frac{1}{|x|}\right), \quad x\in (-\pi,0)\cup(0,\pi], \label{eq:defi_long_den}
\end{equation}
with $d\in(0,1/2)$, and $L$ a slowly varying function defined on $[1/\pi,+\infty)$. (cf. Condition IV in \cite{pipira2018long_range}).
If a (real centered or complex circularly) Gaussian stationary process $(\Xc_t)_{t\in\Z}$ has long memory in this sense, then the LSD of the covariance matrices $T_N[f]$ is not compactly supported and in particular $\lmax(T_N[f])\to\infty$ and \eqref{eq:proc_assym} in Corollary~\ref{th:process} holds. Moreover if $L$ in \eqref{eq:defi_long_den} is quasi-monotone, then the process is also a long memory process in the sense of Definition \ref{defi:long_mem} (see for instance \cite[Corollary~2.2.17]{pipira2018long_range}) with index $\rho=2d-1$. More precisely we have 
$$\gamma(h)\sim 2h^{2d-1}L(h)\Gamma(1-2d)\sin(d\pi) \text{ as } h\to\infty$$
where $\Gamma(t):=\int_0^\infty x^{t-1}e^{-x}\dd x$ denotes the gamma-function. Hence by Theorem~\ref{th:main_toepl}, $T_N$ satisfies the spectral gap condition, and applying Theorem~\ref{th:main_fluct} we get the same result as Corollary~\ref{th:process}.

\subsection{Numerical simulations}
\label{sec:simulations}
\subsubsection*{Simulation 1: Limiting behaviour of $\lmax(S_N)$.}To illustrate Proposition~\ref{prop:main_asympt}, we take 
$$S_N=\frac{1}{n}\Gamma_N^\frac{1}{2}Z_NZ_N^\tran \Gamma_N^\frac{1}{2}$$
with $Z_N$ a $N\times n$ matrix having i.i.d standard real Gaussian entries, and $\Gamma_N=T_N(\gamma)$ is the Toeplitz matrix dertermined by $\gamma(h)=\frac{1}{(1+|h|)^{3/4}}$. Let $N$ take all the values in the finite sequence $\{100,150,200,\dots,3000\}$, and let $n=\lfloor 5N/4\rfloor$. We plot the simulation results in Figure~\ref{fig:con_as}.

\begin{figure}[htbp]
\centering
\subfigure[]{\includegraphics[width=7cm]{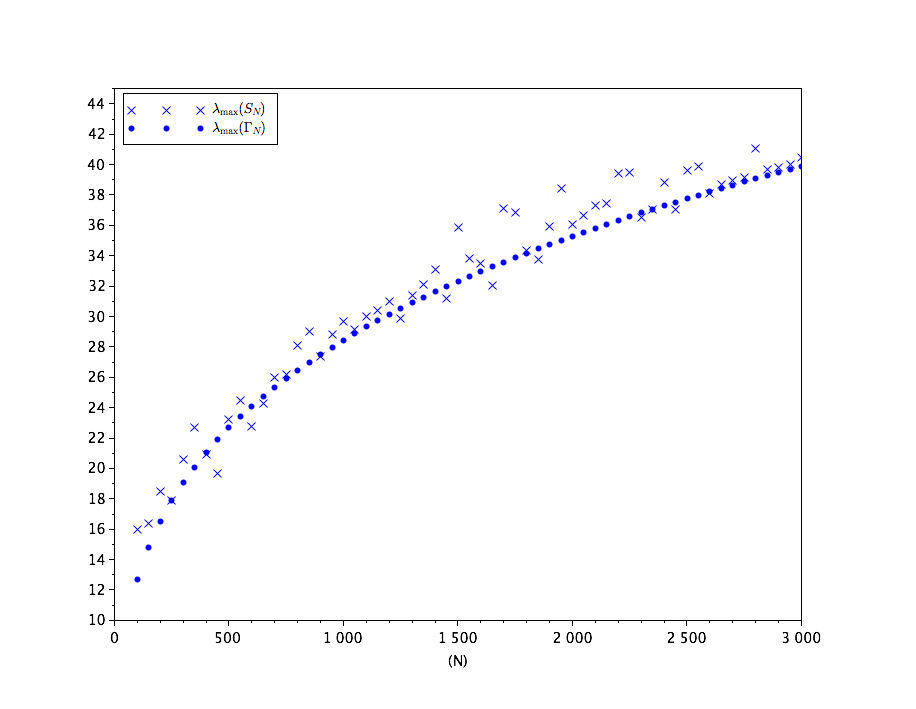}\label{fig:con_as_a}}
\subfigure[]{\includegraphics[width=7cm]{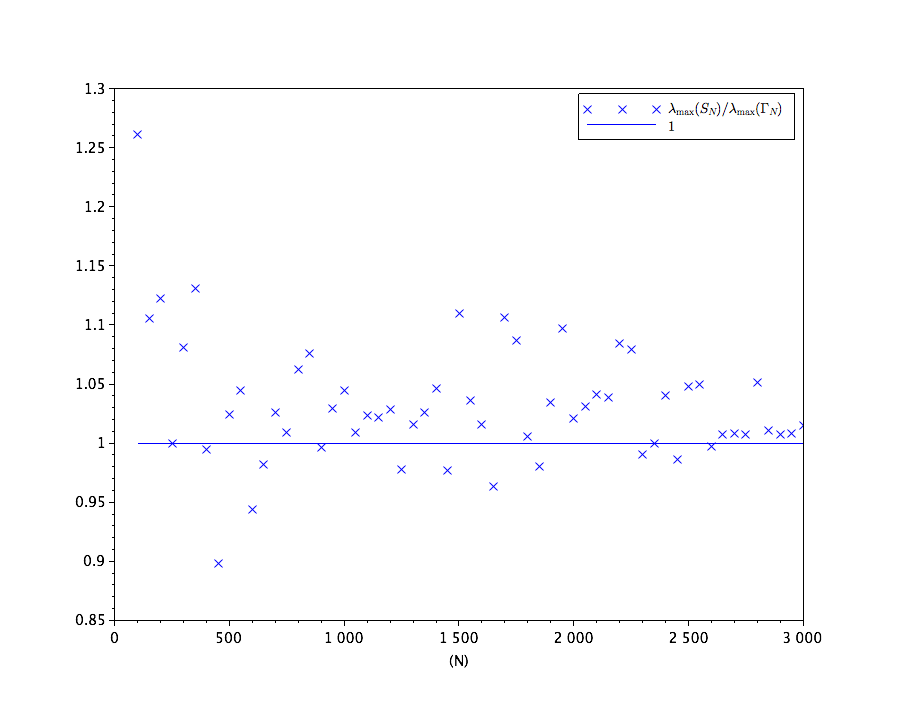}\label{fig:con_as_b}}
\caption{Convergence of ${\lmax(S_N)}/{\lmax(\Gamma_N)}$ to $1$. In \ref{fig:con_as_a}, the values $\lmax(S_N)$ and $\lmax(\Gamma_N)$ are plotted as crosses and solid points respectively; in \ref{fig:con_as_b} the values of the ratio ${\lmax(S_N)}/{\lmax(\Gamma_N)}$ are plotted as crosses, compared with the constant $1$. 
}
\label{fig:con_as}
\end{figure}

\subsubsection*{Simulation 2: Fluctuations of $\lmax(S_N)$.}To illustrate the fluctuations of $\lmax(S_N)$, we fix $N=1000$ and $n=1250$ and let $\Gamma_N=\diag(\lambda_k(T_{1000}(\gamma)))$ with $\gamma$ as in the previous simulation. We take $900$ independant samples of $S_{1000}$, plot the histogram of $F_{1000}$ defined in \eqref{eq:def_FN} and compare with the density of the theoretical limiting law. In Figure~\ref{fig:fluct_a} we simulate the model $S_{1000}$ with $Z_{1000}$ having i.i.d. real Gaussian entries, the limiting law, according to Theorem~\ref{th:main_fluct}, is $\gNr(0,2)$; while in Figure~\ref{fig:fluct_b}, $Z_{1000}$ has i.i.d. standardized exponential entries, i.e. $Z_{i,j}^{(1000)}\sim \mathcal{E}(1)-1$. The limiting law is $\gNr(0,8)$.

\begin{figure}[htbp]
\centering
\subfigure[]{\includegraphics[width=7cm]{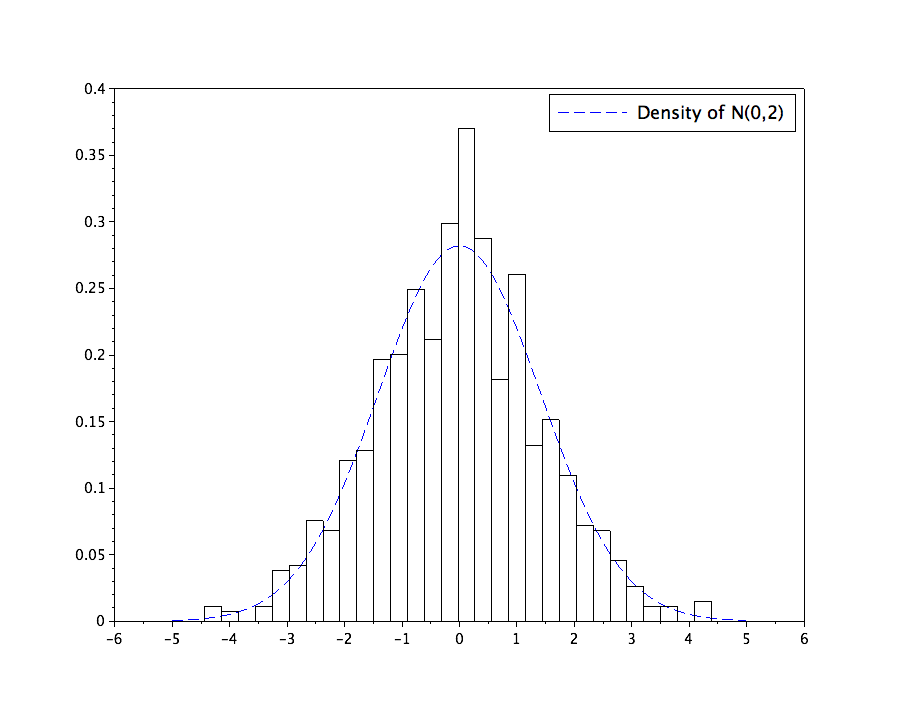}\label{fig:fluct_a}}
\subfigure[]{\includegraphics[width=7cm]{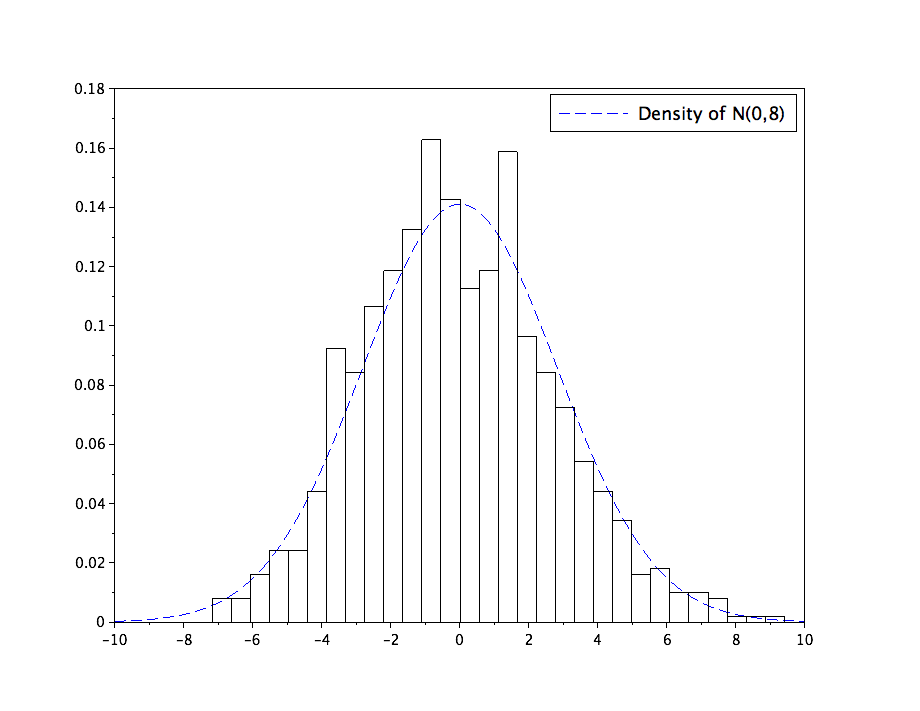}\label{fig:fluct_b}}
\caption{Fluctuations of $\lmax(S_{1000})$, with $Z_{1000}$ having Gaussian entries in \ref{fig:fluct_a} and standardized exponential entries in \ref{fig:fluct_b}. Simulation with $900$ samples.}
\end{figure}

\subsubsection*{Simulation 3: Concentration.} We now address the case $\E|Z_{1,1}^{(N)}|^4=1$. Consider a matrix $Z_N$ with i.i.d. symmetric Bernoulli variables taking values in $\{-1,1\}$. As previously we take $\Gamma_N=\diag(\lambda_k(T_{N}(\gamma)))$ with $\gamma(h)=\frac{1}{(1+|h|)^{3/4}}$. In this case, Theorem \ref{th:main_fluct} asserts that 
$$
\sqrt{n}\left(\frac{\lmax(S_N)}{\lmax(\Gamma_N)}-1-\frac{1}{n}\sum_{k=2}^{N}\frac{\lambda_k(\Gamma_N)}{\lmax(\Gamma_N)-\lambda_k(\Gamma_N)}\right)
\xrightarrow[N,n\to\infty]{\mathcal P} 0\, . 
$$
In Figure~\ref{fig:fluct_bin_a} we plot the fluctuations of $\lmax(S_{1000})$ with $n=1250$ and notice that the corresponding $F_{1000}$ are far more concentrated around $0$ than the previous simulations, as predicted by the theorem. 

An interesting phenomenon occurs in Figure~\ref{fig:fluct_bin_nondiag}, where the same matrix $Z_{1000}$ is considered while we do not diagonalize $\Gamma_{1000}$ and just take $\Gamma_{1000}=T_{1000}(\gamma)$. In this case, the concentration phenomenon disappears, and the obtained histogram is very close to that in Figure~\ref{fig:fluct_a}. This simulation suggests that some universality holds in the fluctuations of $\lmax(S_N)$ if the $(\Gamma_N)$'s are Toeplitz matrices. This will be explored in a forthcoming work.
\begin{figure}[htbp]
\centering
\subfigure[]{\includegraphics[width=7cm]{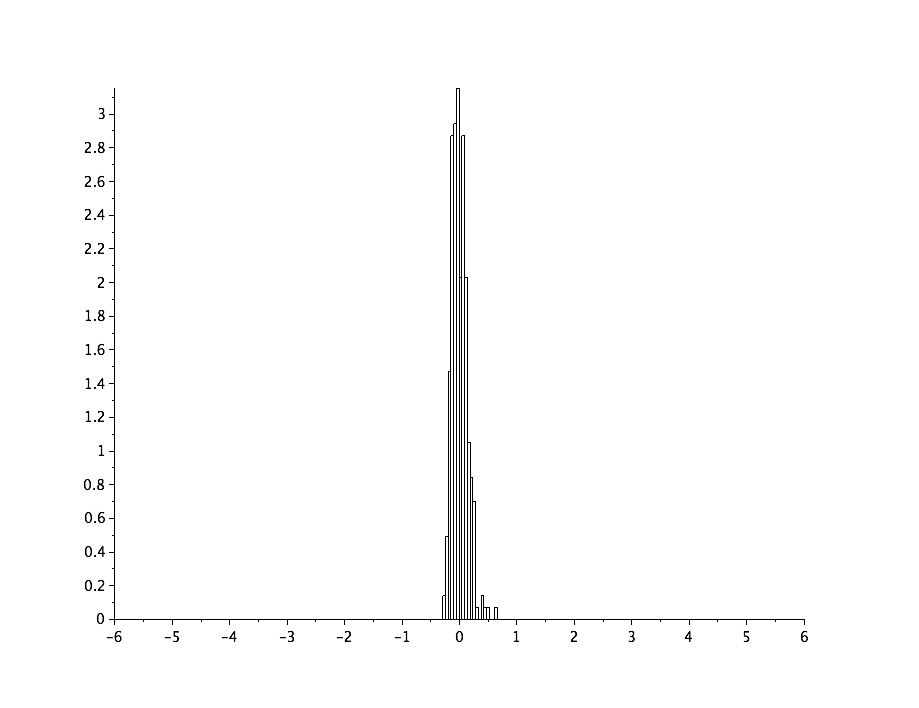}\label{fig:fluct_bin_a}}
\subfigure[]{\includegraphics[width=7cm]{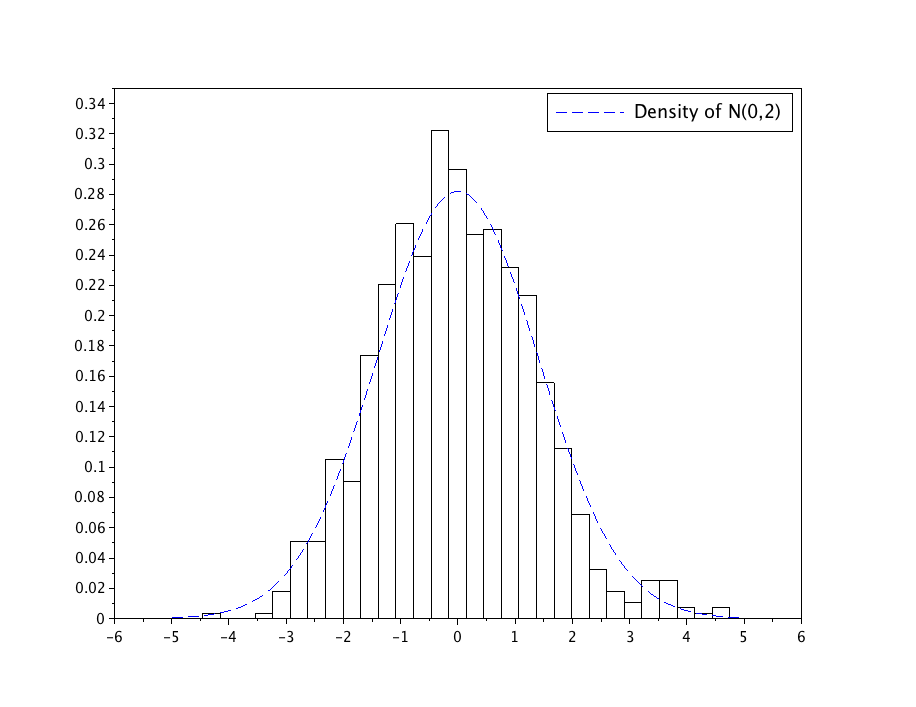}\label{fig:fluct_bin_nondiag}}
\caption{Fluctuations of $\lmax(S_N)$ in the case of symmetric Bernoulli entries. $\Gamma_N$ is diagonal in \ref{fig:fluct_bin_a} and nondiagonal in \ref{fig:fluct_bin_nondiag}.}
\end{figure}

\subsection{Open questions}

\subsubsection*{At the border between long memory and short memory.} An interesting regime is when $\rho=-1$. In this case, the autocovariance function $\gamma(h)=(1+|h|)^{-1}L(h)$ can be summable or not depending on the slowly varying function $L$. For example if $L(h)=\log^{-1-\varepsilon}(2+|h|)$ with $\varepsilon>0$ then $\gamma$ is absolutely summable and the process has short memory.  If $L(\cdot)= 1$ then one can prove that $\lmax(T_N(\gamma))=O(\log N)$, and that the spectral gap condition no longer holds. In this case, the asymptotics \eqref{eq:proc_assym} remains true as
Proposition~\ref{prop:main_asympt} does not rely on the spectral gap condition but only on the condition $\lmax(T_N(\gamma))\to \infty$. The question whether the fluctuations \eqref{eq:proc_fluct} together with their normalization and the limiting distribution hold remains open.

\subsubsection*{Non-Gaussian long memory stationary processes.} A Gaussian long memory stationary process admits a linear representation 
$\Xc_{1:N} = T^{1/2}_N(\gamma) Z_{1:N}$, where $T_N(\gamma)$ is a hermitian Toeplitz matrix and $Z_{1:N}$ is a standard Gaussian vector. This representation 
is key in the analysis of the top eigenvalue of the corresponding large covariance matrix of samples of the process but does not hold anymore if the process 
is not Gaussian. The question whether it is possible to perform the same eigenvalue analysis in the case of non-Gaussian long memory stationary process is open. 

\subsubsection*{Correlation structure of the top eigenvalues.} Beyond the top eigenvalue $\lmax(S_N)$, it would be interesting to understand the asymptotic correlation 
structure and the fluctuations of the (many) largest eigenvalues $(\lmax(S_N), \lambda_2(S_N),\cdots, \lambda_k(S_N))$ for a fixed $k\ge 1$. This question will be addressed in a forthcoming work \cite{tian-non-gaussian}. 

\subsubsection*{Behaviour of the eigenvectors associated to the top eigenvalues.} For bounded spiked models (by bounded we mean $\sup_N\| \Gamma_N\|<\infty$) the structure of the eigenvectors associated to the top eigenvalues has been studied and carries interesting information, see for example \cite{benaych2012singular}. A similar study would be interesting in the general context of unbounded population covariance matrices where $\lmax(\Gamma_N)\to \infty$ and of 
(Gaussian) long memory stationary processes. In this latter case, one needs to have a good understanding of the Toeplitz matrix' $T_N(\gamma)$ eigenvectors.    

\subsubsection*{Universality for non-Gaussian linear stationary processes with long memory.} In the case where $\Gamma_N$ is required to be (block-)diagonal, the variance of the limiting distribution depends on the fourth moment of the entries $Z_{i,j}^{N}$ and may be equal to zero if $\E|Z_{1,1}^{(1)}|^4=1$. 
However when $\Gamma_N$ is a Toeplitz matrix \eqref{eq:autocov} with $\gamma$ satisfying \eqref{eq:intro:long_memory}, this dependence is weakened and Simulation 3 in Section \ref{sec:simulations} strongly suggests that some universality occurs depending on the population covariance matrix $\Gamma_N$, see in particular Figures \ref{fig:fluct_bin_a} and \ref{fig:fluct_bin_nondiag}. This question will be addressed in \cite{tian-non-gaussian}.

\section{Proofs of Proposition~\ref{prop:main_asympt} and Theorem~\ref{th:main_fluct}}
\label{sec:proof_main}

\subsection{A short reminder of results related to large covariance matrices}

Given a probability measure $\mu$ on $\R$, define its Cauchy-Stieltjes transform as 
$$
m_{\mu}(z):=\int\frac{\dd\mu(s)}{z-s}, \quad \forall z\in\C^+:=\{z\in\C:\Im z>0\}
$$
Notice that $m_{\mu}(z)$ is the opposite of the Stieltjes transform $g_{\mu}(z)=\int\frac{\dd\mu(s)}{s-z}$.

For a random matrix $S_N$ given by \eqref{eq:model_main}, we will often consider its companion matrix 
\begin{equation}
\bS_N=\frac{1}{n}Z_N^*\Gamma_N Z_N, \label{eq:def_underS}
\end{equation}
which shares the same non-zero eigenvalues with $S_N$. In particular, $\lmax(\bS_N)=\lmax(S_N)$.  
Recall that
$r_N:=\frac Nn$ and let $\mu^{S_N}, \mu^{\bS_N}$ be the ESD of $S_N$ and $\bS_N$ respectively, then the following relation holds:
\begin{equation}
\mu^{\bS_N}=(1-r_N)\delta_0+r_N\mu^{S_N}. \label{eq:relation_mu_emp}
\end{equation}

\subsubsection*{Limiting spectral distribution.} We recall results from \cite[Theorem 1.1]{silverstein1995strong}. For any probability $\nu$ in $\nonnegR$ and any $r\in(0,+\infty)$, there exists a unique probability measure $\mu=\mu(r,\nu)$ whose Cauchy-Stieltjes transform $m_\mu$ satisfies the equation:
$$
m_\mu(z)=\int\frac{\dd\nu(s)}{z-s(1-r+rzm_\mu(z))}\qquad \textrm{for any}\ z\in \mathbb{C}^+\, .
$$
If the probability measure $\nu$ is the ESD $\mu^A$ associated with a matrix $A$, we simply write $\mu=\mu(r,A)$ instead of $\mu=\mu(r,\mu^A)$. 
Similarly, there exists a unique probability measure $\ulmu=\ulmu(r,\nu)$ with Cauchy-Stieltjes transform $m_{\ulmu}$ satisfying
\begin{equation}\label{eq:fixed-point}
z=\frac{1}{m_{\ulmu}(z)}+r\int\frac{s\dd\nu(s)}{1-sm_{\ulmu}(z)}\qquad \textrm{for any}\ z\in \mathbb{C}^+\, .
\end{equation}
As previously, we will write $\ulmu(r,A)$ instead of $\ulmu(r,\mu^A)$.
If moreover $\mu^{\Gamma_N}\xrightarrow[N\to\infty]{\mathcal D}\nu$, then
\begin{equation}\label{eq:conv-ESD}
\mu^{S_N}\xrightarrow[N,n\to\infty]{\mathcal D}\mu\quad \textrm{a.s.}\qquad \textrm{and}
\qquad 
\mu^{\bS_N}\xrightarrow[N;n\to \infty]{\mathcal D}\ulmu\quad \textrm{a.s.}
\end{equation}
\subsubsection*{Spectrum confinement.} By "spectrum confinement", we refer to the phenomenon where the empirical spectrum of the eigenvalues "concentrates" near the support of the limiting spectral distribution. In the specific case of model \eqref{eq:model_main} under assumption $\sup_N \| \Gamma_N\| <\infty$ and the convergence \eqref{eq:conv-ESD}, spectrum confinement can be roughly expressed (in the absence of spikes) as: for every $\varepsilon>0$, almost surely,
$$
\textrm{supp}(\mu^{S_N})\ \subset\  \textrm{supp}(\mu) + (-\varepsilon, \varepsilon)
$$
for $N$ large enough. 

A more accurate description of spectrum confinement relies on the deterministic equivalent of $\mu^{\bS_N}$ defined as $
\ulmu_N:=\ulmu(r_N,{\Gamma_N})
$ (cf. \eqref{eq:fixed-point} with $\nu=\mu^{\Gamma_N}$). Assume that \ref{ass:infarray} holds. By \cite[Theorem 1.1]{bai1998no}, if there exists $\varepsilon>0$ and an interval $[a,b]$ such that 
$$
[a,b] \cap \left( \textrm{supp}(\ulmu_N)+(-\varepsilon,\varepsilon) \right) = \emptyset
$$ 
for $N$ large enough, then almost surely 
\begin{equation}\label{eq:spectrum-confinement}
\textrm{supp}(\mu^{\bS_N})\ \cap \  [a,b] =\emptyset 
\end{equation}
for $N$ large enough. In particular, if $a>0$ there is no eigenvalue of $S_N$ in $[a,b]$ for $N$ large enough.

The description of the support of a probability distribution defined via a fixed-point equation \eqref{eq:fixed-point} is given in \cite[Theorems 4.1 and 4.2]{silverstein1995analysis}. Based on these results, we now state a necessary and sufficient condition for which a real number $x$ lies outside the support of $\ulmu_N=\ulmu(r_N,{\Gamma_N})$. Let 
\begin{equation}\label{def:BN}
B_N:=\{y\in\R: y\ne 0,\ y^{-1}\ne \lambda_k(\Gamma_N),\ \forall k=1,\dots,N\}\ ,
\end{equation}
and define
\begin{equation}
    x_N(y):=\frac{1}{y}+r_N\int\frac{s\,\dd \mu^{\Gamma_N}(s)}{1-sy}\qquad\textrm{for}\quad  y\in B_N. \label{eq:def_xN}
\end{equation}
A real number $x\in\R$ lies outside the support of $\ulmu_N$ 
if and only if 
$$
\exists\, y\in B_N,\quad x=x_N(y)\quad \textrm{and}\quad x_N'(y)=-\frac{1}{y^2}+r_N\int\frac{s^2\dd\mu^{\Gamma_N}(s)}{(1-sy)^2}\ <\ 0\ .
$$
\subsubsection*{Exact separation.} Let $[a,b]$ be an interval eventually outside the support of $\ulmu_N=\ulmu(r_N,{\Gamma_N})$, assume that $\mu^{\Gamma_N}\rightarrow \nu$ and let $\ulmu=\ulmu(r,\nu)$.  "Exact separation" is a phenomenon that expresses the fact that (almost surely and eventually) the interval $[a,b]$ separates the empirical eigenvalues of matrix $S_N$ exactly in the same proportions as $[1/m_{\ulmu}(a), 1/m_{\ulmu}(b)]$ separates those of matrix $\Gamma_N$. 

This expression has been coined in the article \cite{bai1999exact} by Bai and Silverstein, from which we recall the result of interest to us, that is mainly \cite[Theorem 1.2(2)]{bai1999exact}: Assume in addition to the assumptions of \cite[Theorem~1.1]{bai1998no} (and in particular to assumption \ref{ass:infarray}) that the conditions $m_{\ulmu}(b)>0$ 
and $r(1-\nu(\{0\}))\le 1$ hold.
For $N$ large enough, let $i_N$ be an integer such that 
$$
\lambda_{i_N}(\Gamma_N)>\frac 1{m_{\ulmu}(b)}\qquad\textrm{and}\qquad  \lambda_{i_N+1}(\Gamma_N)<\frac 1{m_{\ulmu}(a)}\, .
$$
Then almost surely,
$
\lambda_{i_N}(S_N)>b$ and $\lambda_{i_N+1}(S_N)<a$
for $N$ large enough. This result will be used in the particular case where $\nu=\delta_0$. In this case, $m_{\ulmu}(x)=\frac 1x$ and $r(1-\delta_0({0}))=0$.

\subsection{Reduction to the bounded model}
When studying $\lmax(S_N)$ under \ref{ass:asympt_spec_st}, the main difficulty is to handle the unboundedness of $\lmax(\Gamma_N)$. In order to circumvent this issue, we define 
$$
\tS_N=\frac{1}{n}\tGamma_N^{1/2}Z_NZ_N^*\tGamma_N^{1/2}\qquad \textrm{where}\quad \tGamma_N:=\frac{\Gamma_N}{\lmax(\Gamma_N)}\ .
$$
In particular, notice that $\lmax(\tS_N)=\frac{\lmax(S_N)}{\lmax(\Gamma_N)}$. Thus, in order to establish the results stated in Proposition \ref{prop:main_asympt} and Theorem \ref{th:main_fluct}, we only need to prove the corresponding results for $\tS_N$. 

Using the definition of $\tGamma_N$, the tightness of ($\mu^{\Gamma_N}$) and the fact that $\lmax(\Gamma_N)\to\infty$, we immediatly deduce the following properties for $\Tilde \Gamma_N$:
$$
\lmax(\tGamma_N)=1\quad \textrm{and}\quad \mu^{\tGamma_N}\xrightarrow[N\to\infty]{\mathcal D} \delta_0\, .
$$
In particular, the spectral norm of $\tGamma_N$ is bounded and many classical results, for instance those of Bai and Silverstein \cite{silverstein1995strong, bai1998no} can be applied to $\tS_N$. Considering this fact, we state and prove below Proposition \ref{prop:main_asympt_bis} and Theorem \ref{th:main_fluct_bis} which are the counterparts of Proposition \ref{prop:main_asympt} and Theorem \ref{th:main_fluct}.

\vspace{.3cm}
\begin{enumerate}[leftmargin=*, label={\bf A2(b)}]
    \item \label{ass:asympt_spec_st_bis} Given a sequence of $N\times N$ positive semidefinite deterministic matrices $\Gamma_N$, the following properties hold:
    $$
    \lmax(\Gamma_N)=1\quad\forall  N\ge 1\quad \textrm{and}\quad  \mu^{\Gamma_N}\xrightarrow[N\to\infty]{\mathcal D} \delta_0\, .
    $$
\end{enumerate}
\vspace{.3cm}
\begin{prop} \label{prop:main_asympt_bis}
Let $S_N$ be a $N\times N$ matrix given by \eqref{eq:model_main} and assume that \ref{ass:model_setting} and \ref{ass:asympt_spec_st_bis} hold. Then
\begin{equation}
\lmax(S_N)\xrightarrow[N,n\to\infty]{\mathcal P} 1. \label{eq:asympt_bis_cv}
\end{equation}
If moreover either the random variables $Z_{ij}^{(N)}$ are standard (real or complex) Gaussian or Assumption \ref{ass:infarray} holds, then the above convergence holds almost surely.
\end{prop}

Although the spectral norm of $\Gamma_N$ is assumed to be bounded while the limiting spectral measure is $\delta_0$, we cannot directly apply the results of \cite{bai2012sample} since we do not assume a clear separation between the spikes and non-spiked eigenvalues. Note that, as quoted in Remark \ref{rem:non-trivial}, all the non-normalized eigenvalues of a Toeplitz matrix converge to infinity at the same rate. Consequently, there is an infinite number of normalized  eigenvalues of a Toeplitz matrix which are generalized spikes in the sense of \cite{bai2012sample}.

\begin{Th}\label{th:main_fluct_bis}
Let $S_N$ be a $N\times N$ matrix given by \eqref{eq:model_main} and assume that \ref{ass:model_setting}, \ref{ass:asympt_spec_st_bis} and \ref{ass:separation} hold.
Assume moreover that one of the following conditions is satisfied:
\begin{enumerate}[label=(\roman*), itemsep=-3pt]
\item Assumption \ref{ass:diagonal} holds,\label{cond:diag}
\item The random variables $Z_{ij}^{(N)}$ are standard complex Gaussian,\label{cond:gauss-C}
\item The random variables $Z_{ij}^{(N)}$ are standard real Gaussian and matrices $\Gamma_N$ are real symmetric, \label{cond:gauss-R}
\end{enumerate}
then
\begin{equation}
    \sqrt{n}\left(\lmax(S_N)-1-\frac{1}{n}\sum_{k=2}^{N}\frac{\lambda_k(\Gamma_N)}{1-\lambda_k(\Gamma_N)}\right)\quad 
    \xrightarrow[N,n\to\infty]{\mathcal D}\quad \gNr(0,\sigma^2), \label{eq:fluct_bis}
\end{equation}
where $\sigma^2=\E|Z_{1,1}^{(1)}|^4-1$.
\end{Th}
In order to prove Proposition  \ref{prop:main_asympt} and Theorem \ref{th:main_fluct}, we only need to apply the above results to $\tS_N$. 

\subsection{Proof of Proposition~\ref{prop:main_asympt_bis}}
We first prove the theorem under assumption \ref{ass:infarray}.
We first establish that
\begin{equation}
\varlimsup_{N,n\to\infty}\lmax(\bS_N)\le 1 \as. \label{eq:limsup}
\end{equation}

Recall the definition of the set $B_N$ in \eqref{def:BN}.
Due to the spectrum confinement property \eqref{eq:spectrum-confinement}, we only need to prove that for any $\varepsilon>0$, the interval $[1+\varepsilon,+\infty)$
eventually stays outside the support of $\ulmu_N=\ulmu(r_N,{\Gamma_N})$. Relying on the caracterization of a point $x$ outside $\textrm{supp}(\ulmu_N)$, this will be the consequence of the following property
$$
\forall\, x\in [1+\varepsilon,\infty),\quad \exists\, y\in B_N,\ x=x_N(y) \quad \textrm{and}\quad x'_N(y)<0
$$
that we now prove.

Since $\lmax(\Gamma_N)=1$ under \ref{ass:asympt_spec_st_bis}, notice that $(0,1)\subset B_N$. Consider a real number $\eta$ such that
$$
\eta \in \left( \frac 1{1+\varepsilon}, 1\right)\ .
$$
For $s\le 1$, we have $|1-s\eta|\ge 1-\eta>0$, therefore by the definition \eqref{eq:def_xN} of $x_N$, and by the fact that $\mu^{\Gamma_N}\xrightarrow[N\to\infty]{} \delta_0$, we have $$
x_N(\eta)\xrightarrow[N,n\to\infty]{}\frac 1\eta <1+\varepsilon\qquad\textrm{and}\qquad x_N'(\eta) \xrightarrow[N,n\to\infty]{} -\frac 1{\eta^2}<0\ .
$$
So for $N$ large enough, we have $x_N(\eta)<1+\varepsilon$, and $x_N'(\eta)<0$. For such $N$'s, note that $x_N$ is continuous on $(0,\eta)$ and that $x_N(y)\to+\infty$ as $y\to 0^+$. We have proved so far that $[1+\varepsilon,\infty)\subset x_N((0,\eta))$. Notice finally that $x'_N$ is increasing on $(0,\eta)$, in particular $x_N'(y)\le x_N'(\eta)<0$ for all $y\in (0,\eta)$. Therefore $x_N((0,\eta))$ and thus $[1+\varepsilon,\infty)$ eventually lie outside the support of $\ulmu_N$. Equation \eqref{eq:limsup} is established.

We now prove that 
\begin{equation}
\varliminf_{N,n\to\infty}\lmax(\bS_N)\ge 1, \as. \label{eq:liminf}
\end{equation}
by an exact separation argument.

As $\ulmu=\delta_0$, we have $1/m_{\ulmu}(a)=a$, $1/m_{\ulmu}(b)=b$ for any $a,b>0$.  We intend to find some constant interval of the form $[a,1-\varepsilon]$, for small $\varepsilon>0$ which separates the eigenvalues of matrix $\Gamma_N$ into two non-empty parts. This is not always possible because even if $\lmax(\Gamma_N)=1$ and $\mu^{\Gamma_N}\rightarrow \delta_0$, there might be some intermediate eigenvalues among the $(\lambda_i(\Gamma_N))$'s for $i\ge 2$ eventually lying in $(0,1)$.
In order to circumvent this issue, we introduce the auxiliary matrices 
$$
\hat{\bS}_N:=\frac{1}{n}Z_N^*\Theta_N Z_N \quad\text{and}\quad\hat{S}_N:=\frac{1}{n}\Theta_N^\frac{1}{2} Z_NZ_N^*\Theta_N^\frac{1}{2}\ ,
$$
where $\Theta_N$ is obtained from the spectral decomposition of $\Gamma_N$ as:
$$
\Theta_N:=U_N\diag(1,0,\cdots)U_N^*\quad \textrm{where}\quad \Gamma_N=U_N\diag(1,\lambda_2(\Gamma_N),\cdots)U_N^*\ .
$$
Using \cite[Theorems 4.1 and 4.2]{silverstein1995analysis}, we conclude that for any $0<\varepsilon<1/2$, the interval $[\varepsilon,1-\varepsilon]$ is eventually outside the support of probability $\ulmu(r_N,\Theta_N)$, obtained from \eqref{eq:fixed-point} with parameters $r_N$ and $\Theta_N$. Notice in particular that  
$$
\lmax(\Theta_N)=1 \qquad \textrm{and}\qquad  \lambda_i(\Theta_N)=0\quad \textrm{for}\quad i=2:N\ . 
$$
Applying \cite[Theorem~1.2]{bai1999exact} to $\hat{\bS}_N$ with separating interval $[\varepsilon,1-\varepsilon]$ for arbitrary $\varepsilon\in (0,1/2)$, we conclude that almost surely, $\lmax(\hat{\underline{S}}_N)>1-\varepsilon$ for $N$ large enough. We have proved so far that 
$$
\varliminf_{N,n\to\infty}\lmax(\hat{\bS}_N)\ge 1
$$ almost surely. Now, since 
$$\bS_N-\hat{\bS}_N=\frac{1}{n}Z_N^*U_N\diag(0,\lambda_2(\bS_N),\cdots,\lambda_N(\bS_N))U_N^*Z_N$$
is nonnegative definite, we have 
$\varliminf_{N,n\to\infty}\lmax(\bS_N)\ge\varliminf_{N,n\to\infty}\lmax(\hat{\underline{S}}_N)\ge 1$.
Therefore Proposition~\ref{prop:main_asympt_bis} with assumption \ref{ass:infarray} is proved.

As a byproduct of the above proof, we can easily prove $\lmax(S_N)\xrightarrow[]{\mathcal P} 1$ without imposing \ref{ass:infarray}. Suppose that $S_N=\frac 1n \Gamma_N^{1/2}Z_NZ_N^*\Gamma_N^{1/2}$ satisfies \ref{ass:model_setting} and \ref{ass:asympt_spec_st_bis} and construct $\check{Z}=(\check{Z}_{i,j})_{i,j\ge 1}$ with $\check{Z}_{i,j}$ i.i.d random variables identically distributed as the entries of $Z_N$. Let $S_N'=\frac 1n \Gamma_N^{1/2}Z'_NZ_N'^*\Gamma_N^{1/2}$ with $Z'_N$ the top-left $N\times n$ submatrix of $\check{Z}$. Then according to the above proof, $\lmax(S_N')$ converges to 1 almost surely, hence in probability. Since $\lmax(S_N)$ and $\lmax(S'_N)$ have the same distribution, we have also $\lmax(S_N)\cvprob 1$.

Finally, we prove that if the entries $Z_{i,j}^{(N)}$ are i.d standard Gaussian variables, and i.i.d for all $1\le i,j\le N$, the convergence \eqref{eq:asympt_bis_cv} holds almost surely without the need of assumption \ref{ass:infarray}. This mainly relies on a concentration argument. Recall that we already have
$\lmax(S_N) \cvprob 1.$
Using \cite[Theorem~5.6]{boucheron2013concentration}, we prove the following concentration inequality: for all $N\ge 1$ and all $\varepsilon>0$,
\begin{equation}
\P(|\sqrt{\lmax(S_N)}-\E\sqrt{\lmax(S_N)}|>\varepsilon)<2e^{-CN\varepsilon^2} \label{eq:concentration}
\end{equation}
where $C>0$ is a proper fixed constant. Indeed it suffices to show that the function $s$ defined by
$$s:Z_N\mapsto \sqrt{\lambda_{\max}(S_N)}=\sqrt{\lambda_{\max}\left(\frac{1}{n}\Gamma_N^\frac{1}{2}Z_NZ_N^*\Gamma_N^\frac{1}{2}\right)}$$
is $\frac{1}{\sqrt{CN}}$-lipschitz, where we consider the $N\times n$ matrix $Z_N$ as a vector in Euclidean space $\R^{Nn}$ when $Z_{i,j}^{(N)}$ are real Gaussian, and in $\R^{2Nn}$ when the entries are complex Gaussian. Note that the Euclidean norm of the vector $Z_N$ is the same as the Frobenius norm $\|Z_N\|_F$ of the matrix $Z_N$. So for any two matrices $Z_N$ and $\hat Z_N$, we have
\begin{align*}
|s(Z_N)-s(\hat Z_N)| &= \frac{1}{\sqrt{n}} \left| \|\Gamma_N^{1/2} Z_N\| - \| \Gamma_N^{1/2} \hat{Z}_N\| \right| &
\le & \ \frac{1}{\sqrt{n}} \|\Gamma_N^{1/2} (Z_N - \hat{Z}_N)\| \\
&\le   \quad \frac{1}{\sqrt{n}}\|\Gamma_N^{1/2}(Z_N-\hat{Z}_N)\|_F & \stackrel{(a)}\le &\ \frac{1}{\sqrt{n}}\|\Gamma_N^{1/2}\|\|Z_N-\hat{Z}_N\|_F \stackrel{(b)}=\frac{1}{\sqrt{n}}\|Z_N-\hat{Z}_N\|_F\, , 
\end{align*}
where $(a)$ follows from the Frobenius norm inequality $\|AB\|_F\le \|A\|\|B\|_F$, and $(b)$ from the fact that $\|\Gamma_N^{1/2}\|=\sqrt{\lmax(\Gamma_N)}=1$.
Thus $s$ is  $1/\sqrt{n}$-lipschitz, and the concentration inequality \eqref{eq:concentration} is proved. Using Borel-Cantelli lemma, we have
\begin{equation}
\sqrt{\lmax(S_N)}-\E\sqrt{\lmax(S_N)}\to 0 \as.\label{eq:conv_sqrt_lambda}
\end{equation}
Together with
$\sqrt{\lmax(S_N)}\xrightarrow{\mathcal P} 1,$
we then obtain that $\E\sqrt{\lmax(S_N)}\to 1$. By \eqref{eq:conv_sqrt_lambda} again, it follows that 
$$\lmax(S_N)\xrightarrow[N,n\to\infty]{} 1\as.$$
The proof of Proposition ~\ref{prop:main_asympt_bis} is complete.

\subsection{Proof of Theorem~\ref{th:main_fluct_bis}} \label{sec:proof_fluct}
We first prove the fluctuation of $\lmax(S_N)$ under \ref{ass:model_setting}, \ref{ass:asympt_spec_st_bis}, \ref{ass:separation} and \ref{ass:diagonal}. Under these assumptions, $\Gamma_N$ is of the form
$$\Gamma_N=\begin{pmatrix}1 & 0 \\ 0 & \Gammasf_{N-1} \end{pmatrix},$$
where $(\Gammasf_{N-1})$ is a sequence of semidefinite positive Hermitian matrices satisfying $\mu^{\Gammasf_{N-1}}\cvweak\delta_0$, and 
$$
\varlimsup_{N\to\infty}\lmax(\Gammasf_{N-1})=\varlimsup_{N\to\infty}\lambda_2(\Gamma_N)=\varlimsup_{N\to\infty}\frac{\lambda_2(\Gamma_N)}{\lambda_1(\Gamma_N)}<1
$$
by assumption \ref{ass:separation}.
We set $d=\varlimsup_{N\to\infty}\lmax(\Gammasf_{N-1})$. For convenience, in this section we omit all the subscript $N$ of matrices, for example we write $S=n^{-1}\Gamma^\frac{1}{2}ZZ^*\Gamma^\frac{1}{2}$. In the following of this section we write $\lmax(S_N)$ as $\lmax$ if it does not cause any ambiguity.

Recall the submatrix notations introduced in \eqref{eq:def_submatrix} and consider the following block decomposition of matrix $S$:
\begin{equation}\label{eq:bloc-S}
S=\begin{pmatrix}
S_{1,1} & S_{1,2:N} \\ 
S_{2:N,1} & S_{2:N,2:N}
\end{pmatrix}=\frac{1}{n}\begin{pmatrix}
Z_{1,\cdot}Z_{1,\cdot}^* & Z_{1,\cdot}Z_{2:N,\cdot}^*\Gammasf^{\frac{1}{2}} \\ 
\Gammasf^{\frac{1}{2}}Z_{2:N,\cdot}Z_{1,\cdot}^* & \Gammasf^{\frac{1}{2}}Z_{2:N,\cdot}Z_{2:N,\cdot}^*\Gammasf^{\frac{1}{2}}
\end{pmatrix}\ .
\end{equation}
Denote
$$
\bS_{2:N,2:N}:=\frac 1nZ_{2:N,\cdot}^*\Gammasf Z_{2:N,\cdot}\qquad \textrm{and}\qquad \mathsf r_N=\frac{N-1}n\ . 
$$
Analog to $\ulmu(r_N,{\Gamma_N})$ defined in \eqref{eq:fixed-point}, we define the probability measure $\ulmu(\mathsf r_N,{\Gammasf_{N-1}})$ whose Cauchy-Stieltjes transform $\underline{\mathsf{m}}$ satisfies the equation 
$$z=\frac{1}{\underline{\mathsf m}(z)}+\frac{1}{n}\sum_{k=2}^{N}\frac{\lambda_k(\Gamma_N)}{1-\lambda_k(\Gamma_N)\underline{\mathsf{m}}(z)},\quad \forall z\in\C^+.$$
Also, for all $y\in \mathsf B_N:=\{y\in\R: y\ne 0, y^{-1}\ne \lambda_k(\Gamma_N), \forall k=2,\dots,N\}$, we define
\begin{equation}
    \mathsf x_N(y):=\frac{1}{y}+\frac{1}{n}\sum_{k=2}^{N}\frac{\lambda_k(\Gamma_N)}{1-\lambda_k(\Gamma_N)y}=\frac{1}{y}+\mathsf r_N\int\frac{s}{1-sy}\dd \mu^{\Gammasf_{N-1}}(s).\label{eq:def_sfxN}
\end{equation}
Consider in particular
\begin{equation}\label{def:theta}
\theta_N=\mathsf x_N(1) = 1 + \frac 1n \sum_{k=2}^N \frac{\lambda_k(\Gamma_N)}{1-\lambda_k(\Gamma_N)}\ .
\end{equation}

Let $\varepsilon>0$ be small enough. Thanks to the assumption $\mu^{\Gammasf_{N-1}}\cvweak\delta_0$ and $d=\varlimsup\lmax(\Gammasf_{N-1})<1$, one can adapt the first part of the proof of Proposition~\ref{prop:main_asympt_bis} to obtain that eventually
$$
\sup \supp\ulmu(\mathsf r_N,{\Gammasf_{N-1}})\ <\ d+\varepsilon\, .
$$
Let $\varepsilon<\frac {1-d}2$ so that $d+\varepsilon<1-\varepsilon$ and consider the family of events $\Omega_N$ defined as
\begin{equation}
\Omega_N:=\{\lambda_\mathrm{max}(S_{2:N,2:N})<d+\varepsilon<1-\varepsilon<\lmax\}. \label{eq:def_Omega_N}
\end{equation}
According to the spectrum confinement property \cite[Theorem 1.1]{bai1998no} and to Proposition~\ref{prop:main_asympt_bis}, one has
$$
\P(\Omega_N)\xrightarrow[N,n\to\infty]{} 1\ .
$$
In particular, for any sequence of events $A_N$, we have $\P(A_N)-\P(A_N \cap \Omega_N)\to_{N\to\infty} 0$, which can be written
$$
\P(A_N)\doteq\P_\Omega(A_N)
$$ if one writes $\P_\Omega(\cdot)$ for $\P(\,\cdot\, \cap \Omega_N)$ and recall the notation $x\doteq y$ for $x-y\to 0$.
Hence, with no loss of generality, we will assume below that $\Omega_N$ holds. 

Let $\lmax \in \Omega_N$. Using the block decomposition \eqref{eq:bloc-S} of $S$ together with the determinantal formula based on Schur complements (see for instance \cite[Section 0.8.5]{horn2012matrix}), the eigenvalue $\lmax$ satisfies the equation: 
\begin{multline}
\det(\lmax I-S)\\
=\quad\left(\lmax I-S_{1,1}-S_{1,2:N}(\lmax I-S_{2:N,2:N})^{-1}S_{2:N,1}\right)\det(\lmax I-S_{2:N,2:N})\quad=\quad0\, .
\end{multline}
Since $\det(\lmax I-S_{2:N,2:N})\ne 0$ on $\Omega_N$, we have
\begin{eqnarray}
    \lmax&=&S_{1,1}+S_{1,2:N}(\lmax I-S_{2:N,2:N})^{-1}S_{2:N,1}\, , \nonumber\\
    &=&  \frac{Z_{1,\cdot}Z_{1,\cdot}^*}n + \frac{Z_{1,\cdot}Z_{2:N,\cdot}^*\Gammasf^{\frac{1}{2}}}n 
    \left( \lmax I - \frac 1n \Gammasf^{\frac{1}{2}}Z_{2:N,\cdot}Z_{2:N,\cdot}^*\Gammasf^{\frac{1}{2}}\right)^{-1} \frac{\Gammasf^{\frac{1}{2}}Z_{2:N,\cdot}Z_{1,\cdot}^*}n\, , \nonumber\\
    &=& \frac{Z_{1,\cdot}} n\Big( I +  A^*
    \left( \lmax I -  A A^*\right)^{-1}
A
    \Big)  Z_{1,\cdot}^*\, . \qquad \left( A=n^{-1/2}\,  {\Gammasf^{\frac{1}{2}}Z_{2:N,\cdot}}\right)
      \label{eq:lambda_N}
\end{eqnarray}
Using the equality $I+A^*(\lambda I-AA^*)^{-1}A=\lambda(\lambda I-A^*A)^{-1}$ for all scalar $\lambda$ and all matrix $A$ such that $\lambda I-AA^*$ and $\lambda I-A^*A$ are invertible, the equation \eqref{eq:lambda_N} is equivalent to
\begin{equation}
    1=\frac{1}{n}Z_{1,\cdot}(\lmax I-\bS_{2:N,2:N})^{-1}Z_{1,\cdot}^*\, .\label{eq:lambda_N_2}
\end{equation}
As $\theta_N=\mathsf x_N(1)\ge 1$ lies outside the support of $\ulmu(\mathsf r_N,{\Gammasf_{N-1}})$ for large $N$, \cite[Theorem~4.2]{silverstein1995analysis} yields
$$
\underline{\mathsf m}(\theta_N)=\underline{\mathsf m}(\mathsf x_N(1))=1\, .
$$
This can be regarded as a ``deterministic" version of \eqref{eq:lambda_N_2}, which indicates that $\lmax$ and $\theta_N$ are comparable. 

In order to prove the Gaussian fluctuations of $\lmax$, we need to prove that for all $b\in \R$
\begin{equation}
\P\left(\lmax\le \eta_N\right)\xrightarrow[N,n\to\infty]{} \Phi(b,\sigma) 
\label{eq:cv_proba}
\end{equation}
where 
$$
\eta_N:=\theta_N+\frac b{\sqrt{n}}\,  ,\qquad \Phi(x,\sigma):=\frac{1}{\sigma\sqrt{2\pi}}\int_{-\infty}^x e^{-\frac{t^2}{2\sigma^2}}\dd t\qquad \textrm{and}
\qquad \sigma^2=\E\left|Z^{(1)}_{1,1}\right|^4-1\, .
$$
Note that on $\Omega_N$ the function
$$
\Upsilon(\lambda):=\frac{1}{n}Z_{1,\cdot}(\lambda I-\bS_{2:N,2:N})^{-1}Z_{1,\cdot}^*
$$
is decreasing on $(d+\varepsilon,+\infty)$. Let $N$ large enough so that $\eta_N>d+\varepsilon$.

Taking into account the fact that $\Upsilon(\lmax)=1$ due to \eqref{eq:lambda_N_2},
we have
\begin{equation}
\P_\Omega(\lmax\le \eta_N)\ =\ \P_\Omega\left(\Upsilon(\eta_N) \le 1\right) 
\ =\ \P_\Omega\left(\sqrt{n} \left( \Upsilon(\eta_N) - \underline{\mathsf m}(\eta_N) \right)
\le \sqrt{n}(1 - \underline{\mathsf m}(\eta_N))\right)\ .
\label{eq:proba_omega_ineq}
\end{equation}
We first prove that
\begin{equation}\label{eq:DL-m}
\sqrt{n}(1- \underline{\mathsf m}(\eta_N)) = b +o(1)\ .
\end{equation}
Taking into account the fact that $\underline{\mathsf m}(\theta_N)=1$ and performing a Taylor expansion on $\underline{\mathsf m}$ around $\theta_N$ yields
$$
\sqrt{n}(1- \underline{\mathsf m}(\eta_N)) \ =\  \sqrt{n}(\underline{\mathsf m}(\theta_N)- \underline{\mathsf m}(\eta_N)) \ =\ -b\,\underline{\mathsf m}'(\theta_N)-\underline{\mathsf m}''(\xi_N)\frac{b^2}{\sqrt{n}} 
$$
where $\xi_N$ is between $\theta_N= \mathsf x_N(1)$ and $\eta_N$. The assumptions $\mu^{\Gammasf_{N-1}}\cvweak\delta_0$ and $d=\varlimsup\lmax(\Gammasf_{N-1})<1$ yield
$$
\theta_N,\eta_N\xrightarrow[N,n\to\infty]{} 1\ .
$$
Similarly, one proves that $\mathsf x_N'(1)\xrightarrow[N,n\to\infty]{} -1$. By \cite[Theorem 4.2]{silverstein1995analysis}, equality $\underline{\mathsf m}(\mathsf x_N(y))=y$ holds for any $y\notin\supp\mu^{\mathsf r_N,\Gammasf_{N-1}}$. Differentiating, we get
$$
\underline{\mathsf m}'(\mathsf x_N(y))\mathsf x_N'(y)=1\qquad \textrm{and}
\qquad \underline{\mathsf m}'(\mathsf x_N(1))= \frac 1{\mathsf x_N'(1)}\xrightarrow[N,n\to\infty]{} - 1\ .
$$
Finally, for large $N$, we have 
$
\sup_N\supp\ulmu(\mathsf r_N,\Gammasf_{N-1})<d+\varepsilon<1-\varepsilon<\min(\eta_N,\theta_N)$ which implies 
$$
|\underline{\mathsf m}''(\xi_N)|=2\left| \int\frac{\dd\ulmu(\mathsf r_N,\Gammasf_{N-1})(s)}{(\xi_N-s)^3}\right| \ \le \ \frac 2{(1-d - 2\varepsilon)^3}\ .
$$
Plugging this into the Taylor expansion finally yields \eqref{eq:DL-m}.

We now go back to \eqref{eq:proba_omega_ineq} and handle the quantity $\sqrt{n} ( \Upsilon(\eta_N) - \underline{\mathsf m}(\eta_N))$. More precisely, we prove in the sequel that
\begin{equation}\label{eq:almost-CLT}
\sqrt{n} ( \Upsilon(\eta_N) - \underline{\mathsf m}(\eta_N)) = \sqrt{n} \left( \frac 1n Z_{1,\cdot} Z_{1,\cdot}^* - 1\right) + o_P(1)\ .
\end{equation}
In order to proceed, we need the following estimates, valid under the assumptions of Theorem \ref{th:main_fluct_bis}.
\begin{prop} \label{prop:useful-estimates} Assume that \ref{ass:model_setting}, \ref{ass:asympt_spec_st_bis}, \ref{ass:separation} and \ref{ass:diagonal} hold, then 
\begin{enumerate}
\item[(a)] 
$
\sqrt{n}\left( \frac{1}n \tr\left(\eta_N I-\bS_{2:N,2:N}\right)^{-1}- \underline{\mathsf m}(\eta_N)
\right) \quad \xrightarrow[N,n\to\infty]{\mathcal P}\quad 0\ ,
$
\item[(b)]
$
\frac{\sqrt{n}}{\eta_N}\left(\frac 1n Z_{1,\cdot}(\eta_N I-\bS_{2:N,2:N})^{-1}\bS_{2:N,2:N}Z_{1,\cdot}^*
-\frac 1n \tr (\eta_N I-\bS_{2:N,2:N})^{-1}\bS_{2:N,2:N}\right) 
\quad \xrightarrow[N,n\to\infty]{\mathcal P}\quad 0\ .
$
\end{enumerate}
\end{prop}
Proof of Proposition \ref{prop:useful-estimates} is postponed to Section \ref{sec:proof-prop}. We have
\begin{eqnarray}
\lefteqn{\sqrt{n} ( \Upsilon(\eta_N) - \underline{\mathsf m}(\eta_N))}\nonumber\\
&=& \sqrt{n} \left( \Upsilon(\eta_N) - \frac 1n \tr \left( \eta_N I- \bS_{2:N,2:N} \right)^{-1} + 
\frac 1n \tr \left( \eta_N I - \bS_{2:N,2:N} \right)^{-1} -
\underline{\mathsf m}(\eta_N)\right)\ ,\nonumber \\
&=& \sqrt{n} \left( \Upsilon(\eta_N) - \frac 1n \tr \left( \eta_N I- \bS_{2:N,2:N} \right)^{-1}\right) +o_P(1)
\label{eq:modif_3}
\end{eqnarray}
by the first part of Proposition \ref{prop:useful-estimates}. We now apply the resolvent identity $A^{-1}-B^{-1}=A^{-1}(B-A)B^{-1}$ to $A=\eta_N I-\bS_{2:N,2:N}$ and $B=\eta_NI$ and obtain
\begin{eqnarray}
\lefteqn{\sqrt{n}\left( \Upsilon(\eta_N)-\frac{1}{n}\tr\left(\eta_N I-\bS_{2:N,2:N}\right)^{-1} \right)}\nonumber\\
 &= & \frac{\sqrt{n}}{\eta_N}\left(\frac 1n  Z_{1,\cdot}(\eta_N I-\bS_{2:N,2:N})^{-1}\bS_{2:N,2:N}Z_{1,\cdot}^*-\frac 1n \tr (\eta_N I-\bS_{2:N,2:N})^{-1}\bS_{2:N,2:N}\right) 
 \nonumber \\
 &&\qquad +\frac{\sqrt{n}}{\eta_N}\left(\frac 1n Z_{1,\cdot}Z_{1,\cdot}^*-1\right)\ , \nonumber \\
 &= & \frac{\sqrt{n}}{\eta_N}\left(\frac 1n Z_{1,\cdot}Z_{1,\cdot}^*-1\right) + o_P(1)\label{eq:modif_4}
\end{eqnarray}
where the last equality follows from the second estimate of Proposition \ref{prop:useful-estimates}. Notice that by the standard Central Limit theorem, 
$$
\sqrt{n}\left( \frac{1}{n}Z_{1,\cdot}Z_{1,\cdot}^*-1\right)=\sqrt{n}\left( \frac 1n \sum_{j=1}^n|Z_{1,j}|^2-1\right)\quad \xrightarrow[n\to\infty]{\mathcal D}
\quad  
\gNr(0,\var|Z_{1,1}|^2)
$$
where $\var|Z_{1,1}|^2=\E |Z_{1,1}|^4 -1$. Since $\eta_N\to 1$, one has 
$$
\frac{\sqrt{n}}{\eta_N}\left(\frac 1n Z_{1,\cdot}Z_{1,\cdot}^*-1\right) = \sqrt{n}\left(\frac 1n Z_{1,\cdot}Z_{1,\cdot}^*-1\right) +o_P(1)\, .
$$
Plugging this last estimate into \eqref{eq:modif_4} and \eqref{eq:modif_3} finally yields \eqref{eq:almost-CLT}. We can now conclude the proof of the CLT:
\begin{eqnarray}
\P(\lmax\le \eta_N)&\doteq&
\P_\Omega(\lmax\le \eta_N)\ ,\nonumber \\
&\stackrel{(a)}=& \P_\Omega\left(\sqrt{n} \left( \Upsilon(\eta_N) - \underline{\mathsf m}(\eta_N) \right)
\le b+o(1)\right)\ ,\nonumber \\
&\stackrel{(b)}= & \P_\Omega\left(  \sqrt{n} \left( \frac 1n Z_{1,\cdot} Z_{1,\cdot}^* - 1\right) +o_P(1) \le b \right) \ ,\nonumber \\
&\doteq& \P\left(  \sqrt{n} \left( \frac 1n Z_{1,\cdot} Z_{1,\cdot}^* - 1\right) +o_P(1) \le b \right)\label{eq:almost-CLT-bis}
\end{eqnarray}
where $(a)$ follows from \eqref{eq:proba_omega_ineq} and \eqref{eq:DL-m} and $(b)$ follows from \eqref{eq:almost-CLT}. 
We can now get rid of the term $o_P(1)$ in \eqref{eq:almost-CLT-bis} by Slutsky's theorem and finally obtain the desired result:
$$
\P(\sqrt{n} ( \lmax - \theta_N) \le b) \ =\ \P (\lmax \le \eta_N)\quad\xrightarrow[N,n\to\infty]{}\quad  \Phi(b,\sigma)\, , \qquad \sigma^2=\E|Z_{1,1}|^4-1\, .
$$
This completes the proof of Theorem \ref{th:main_fluct_bis} under condition \ref{cond:diag}.

Assume now that $Z_{ij}^{(N)}\sim {\mathcal N}_\C(0,1)$ and consider the eigen-decomposition $\Gamma_N=U_ND_NU_N^*$, where $U_N$ is unitary and $D_N=\diag(\lambda_1(\Gamma_N),\dots,\lambda_N(\Gamma_N))$. Then $S_N$ can be written as
$$
S_N=\frac{1}{n}U_ND_N^{\frac{1}{2}}\left(U_N^*Z_N\right)\left(Z_N^*U_N\right)D_N^{\frac{1}{2}}U_N^* = \frac{1}{n}U_ND_N^{\frac{1}{2}}\tilde Z_N \tilde Z_N^*D_N^{\frac{1}{2}}U_N^*
\quad \textrm{where}\quad \tilde Z_N= U_N^*Z_N
$$
and has the same eigenvalues as the matrix 
$R_N=n^{-1}D_N^{\frac{1}{2}}\tilde Z_N \tilde Z_N^* D_N^{\frac{1}{2}}$. It remains to notice that $\tilde Z_N$ has i.i.d. ${\mathcal N}_\C(0,1)$ entries. In particular,
$R_N$ satisfies \ref{ass:model_setting}, \ref{ass:asympt_spec_st_bis}, \ref{ass:separation} and \ref{ass:diagonal}, and the desired result follow for $S_N$.
Theorem \ref{th:main_fluct_bis} is established under condition \ref{cond:gauss-C}.

Assume now that $Z_{ij}^{(N)}\sim {\mathcal N}_\R(0,1)$ and that $\Gamma_N$ is real symmetric. In this case, $\Gamma_N$'s eigen-decomposition writes $\Gamma_N= O_N D_N O_N^\tran$, where matrix $O_N$ is orthogonal. It remains to notice that $O_N^\tran Z_N$ has i.i.d ${\mathcal N}_\R(0,1)$ entries and to proceed as in the complex case to prove Theorem \ref{th:main_fluct_bis} under condition \ref{cond:gauss-R}.

Proof of Theorem \ref{th:main_fluct_bis} is completed.
\subsubsection{Proof of Proposition \ref{prop:useful-estimates}}\label{sec:proof-prop}
We first establish item $(a)$. Denote by 
$$
\Delta_N(x)= \frac{1}{n} \tr\left(x I-\bS_{2:N,2:N}\right)^{-1}- \underline{\mathsf m}(x)\ .
$$ 
We will first establish that $n\Delta_n(\eta_N)$ is tight and then, as an easy consequence, we will deduce the desired convergence: $\sqrt{n}\Delta_N(\eta_N)\xrightarrow[N,n\to\infty]{\mathcal P} 0$. 

If $x\ge 1-\varepsilon$ is fixed with $1-\varepsilon> d+\varepsilon$, then the tightness of $n\Delta_N(x)$ is a consequence of Bai and Silverstein's peripheral results of their CLT paper \cite{bai2004clt}, see also \cite[Chapter 9]{bai2010spectral}. In fact,
$$
\Delta_N(x) = \int f(x,\lambda) \mu^{\bS_{2:N,2:N}}(d\lambda) - \int  f(x,\lambda)  \ulmu(\mathsf r_N,\Gammasf_{N-1})(d\lambda)\quad \textrm{where} \quad f(x,\lambda)=\frac 1{x-\lambda}\ .
$$
Notice that for any $x\ge 1-\varepsilon$, $\lambda\mapsto f(x,\lambda)$ is analytic in a neighbourhood of $[0,d+\varepsilon]$ which contains the support of $\ulmu(\mathsf r_N,{\Gammasf_{N-1}})$. According to \cite[Theorem 9.10(1)]{bai2010spectral} and to the remark at the end of page 265 in \cite{bai2010spectral} which tightens the interval where the function $f(x,\cdot)$ needs to be analytic, we immediatly obtain the tightness of $(n\Delta_N(x))$.

The case where $x=\eta_N\ge 1-\varepsilon$ for $N$ large necessitates some adaptation. We closely follow \cite[Chapter 9]{bai2010spectral}. Denote by 
$$
M_N(z) = n \left( m^{\bS_{2:N,2:N}}(z) - \underline{\mathsf m}(z)\right)
$$
and by ${\mathcal C}^+$ the contour defined by ($\delta,u>0$ fixed)
$$
{\mathcal C}^+={\mathcal C}_\ell \cup {\mathcal C}_{\textrm{up}} \cup {\mathcal C}_r\quad \textrm{where} \quad 
\left\{ 
\begin{array}{ccc}
{\mathcal C}_\ell&=& \{z=(-\delta, y)\, ,\ y\in [0,u]\} \\
{\mathcal C}_{\textrm{up}} &=&  \{z=(x, u)\, ,\ x\in [-\delta,d+\varepsilon]\} \\
{\mathcal C}_r &=&  \{z=(d+\varepsilon, y)\, ,\ y\in [0,u]\}
\end{array}
\right.\ .
$$
Consider the truncated version $\hat{M}_N(z)$ of $M_N(z)$ as defined in \cite[(9.8.2)]{bai2010spectral} then
$$
\int_{\mathcal C} \frac 1{\eta_N-z}\,  \left( \hat M_N(z) - M_N(z) \right) \, \dd z \xrightarrow[N,n\to\infty]{\as.} 0 \qquad \textrm{where} \qquad  {\mathcal C} = {\mathcal C}^+\cup \overline{{\mathcal C}^+}
$$
and $\{ \hat M_N(\cdot) \}$ forms a tight sequence on ${\mathcal C}$. Consider now the mapping 
$$
\Gamma_N: \hat M_N(\cdot) \quad \longmapsto \quad \frac 1{2\ii \pi} \int_{\mathcal C} \frac 1{\eta_N -z}\,  \hat M_N(z)\, \dd z\ .
$$
$\Gamma_N$ is a continuous mapping from $C({\mathcal C}, \R^2)$ to $\C$. Applying Prohorov's theorem (see for instance \cite[Theorem 16.3]{kallenberg2006foundations}) and the continuous mapping theorem \cite[Theorem 4.27]{kallenberg2006foundations}, we conclude that $\Gamma_N(\hat M_N)$ is tight. It remains to notice that 
$$
n\Delta_N(\eta_N) = \Gamma_N(\hat M_N) + \underbrace{\left( \Gamma_N(\hat M_N) - \Gamma_N(M_N)\right)}_{\to 0\ \as.}
$$ 
to conclude that $n\Delta_N(\eta_N)$ is tight. Now let $\delta>0$ be fixed, then
$$
\P ( | \sqrt{n} \Delta_N(\eta_N)| > \delta) = \P (|n \Delta_N(\eta_N)| > \sqrt{n} \delta) \xrightarrow[N,n\to\infty]{} 0
$$
by tightness, hence the convergence of $\sqrt{n}\Delta_N(\eta_N)$ to zero in probability. Part $(a)$ of Proposition \ref{prop:useful-estimates} is proved.

We now prove part $(b)$ of Proposition \ref{prop:useful-estimates} and rely on the lemma on quadratic forms \cite[Lemma B.26]{bai2010spectral}. Denote by
$$
P_N = \sqrt{n} \left( 
\frac 1n Z_{1,\cdot}(\eta_N I-\bS_{2:N,2:N})^{-1}\bS_{2:N,2:N}Z_{1,\cdot}^*-\frac 1n \tr\left\{ (\eta_N I-\bS_{2:N,2:N})^{-1}\bS_{2:N,2:N}\right\} \right)$$
and apply the lemma on quadratic forms with $p=2$: There exists a constant $C$ such that
$$
\E_{\mathsf Z_{2:N,\cdot}}(|P_N|^2)\quad \le\quad  \frac{C}{n}\E(|Z^{(1)}_{1,1}|^4)\,\tr \left\{ (\eta_N I-\bS_{2:N,2:N})^{-2}\,\bS_{2:N,2:N}^2\right\} \ .
$$
Taking into account the facts that
$$
\lim_{n\to\infty} \eta_N =1\ ,\quad \varlimsup_N\lmax(\bS_{2:N,2:N})\stackrel{\mathcal P}\le d\qquad \textrm{and}\qquad \mu^{\bS_{2:N,2:N}}
\xrightarrow[N,n\to\infty]{\mathcal D} 
\delta_0\ \as.\ ,
$$
we obtain that 
$$
\frac{1}{n}\tr \left\{ (\eta_N-\bS_{2:N,2:N})^{-2}\bS_{2:N,2:N}^2\right\} =\int\frac{s^2}{(\eta_N-s)^2}\, \mu^{\bS_{2:N,2:N}} (\dd s)\quad 
\xrightarrow[N,n\to\infty]{\mathcal P} 
\quad 0\ .
$$
Thus $\E_{Z_{2:N,\cdot}}(|P_N|^2)$ converges to zero in probability, from which we deduce that for $\delta>0$,
$$
\E_{Z_{2:N,\cdot}}(\ind_{|P_N|^2>\delta}) \, \le \, \frac 1{\delta} \E_{Z_{2:N,\cdot}} \left( |P_N|^2 \ind_{|P_N|^2>\delta}\right)
 \xrightarrow[N,n\to\infty]{\mathcal P} 0\, .
$$ 
Finally 
$$
\P(|P_N|^2>\delta)=\E\, \E_{Z_{2:N,\cdot}}(\ind_{|P_N|^2>\delta})\xrightarrow[N,n\to\infty]{} 0\, ,
$$
which completes the proof of Proposition \ref{prop:useful-estimates}.

\section{Proof of Theorem~\ref{th:main_toepl}}
\label{sec:proof_toepl}
In order to study the spectral gap associated to the family of Toeplitz matrices and to prove Theorem~\ref{th:main_toepl}, we follow the method used in 
\cite{bogoya2012eigenvalues}. The main idea is to interpret the eigenvalues of the Toeplitz matrix $T_N$ as eigenvalues of an operator $\K_N$ using Widom-Shampine's Lemma, and then analyse the convergence of this operator, correctly normalized.

In this section, for $p\in[1,\infty]$, the $L^p$ norm of a function $f$ is denoted by $\|f\|_p$, and the $L^p\rightarrow L^p$ norm of an operator $\K$ is denoted by $\|\K\|_p$. Recall that $\|\K\|_p:=\sup_{\|f\|_p=1}\|\K f\|_p$.

\subsection{Widom-Shampine's Lemma and convergence of operators} \label{subsec:toepl_conv}
We first recall Widom-Shampine's Lemma, see \cite{bogoya2012eigenvalues} for a proof.
\begin{lemma}[Widom-Shampine] \label{lem:widom_shamp} Let $A=(a_{i,j})_{i,j=0}^{N-1}$ be a matrix with complex entries $a_{i,j}$, and let $G$ be the integral operator on $L^2(0,1)$ defined by
$$(Gf)(x)=\int_0^1 a_{\lfloor Nx\rfloor \lfloor Ny\rfloor}f(y)\dd y, \quad x\in(0,1).$$
Then a nonzero complex number $\lambda$ is an eigenvalue of $A$ of a certain algebraic multiplicity if and only if $\lambda/N$ is an eigenvalue of $G$ of the same algebraic multiplicity.
\end{lemma}

Let $c=(c_k)_{k\in\Z}$ be the sequence in Theorem~\ref{th:main_toepl}, and $\rho\in(-1,0)$ be the index, then the function $R(h):=c_{\lfloor |h|\rfloor}$ is even 
and regularly varying and $R(k)=c_k$. By Definition \ref{defi:regular_slow}, $R(N)\ne 0$ for large enough $N\in\N$, for convenience we can suppose that 
$R(N)\ne 0$ for all $N\in\N$ without loss of generality. By Widom-Shampine's Lemma, for each $N$, the matrix $T_N(c)/(NR(N))$ has the same nonzero 
eigenvalues (with the same multiplicities) as the integral operator $\K_N^{(\rho)}$ defined on $L^2(0,1)$ by
\begin{equation}(\K_N^{(\rho)}f)(x)=\int_0^1 \frac{R(\lfloor Nx\rfloor-\lfloor Ny\rfloor)}{R(N)}f(y)\dd y\,. \label{eq:def_Kappa_N}
\end{equation}
We will prove that the operators $\K_N^{(\rho)}$ converge in the operator norm to the operator $\K^{(\rho)}$ defined on $L^2(0,1)$ by
\begin{equation}
(\K^{(\rho)}f)(x)=\int_0^1 |x-y|^\rho f(y)\dd y. \label{eq:def_Kappa}
\end{equation}
For this we need the following Lemma \ref{lem:unif_con} which is a special case of the uniform convergence theorem of regularly varying functions.
\begin{lemma}[{\cite[Theorem~1.5.2]{bingham1989regular}}]
\label{lem:unif_con}If $R$ is regularly varying with index $\rho<0$, then for every $a>0$
$$
\sup_{x>a} \left| \frac{R(xy)}{R(y)} -  x^\rho\right| \xrightarrow[y\to\infty]{} 0\, .
$$
\end{lemma}
The following description of the asymptotic integral of regularly varying functions will also be useful in the sequel. 
\begin{lemma}[{\cite[Proposition~1.5.8]{bingham1989regular}}]\label{lem:equi_integ} If $R$ is regularly varying with index $\rho>-1$, and suppose that $R$ is locally bounded, then
$$\int_0^y R(x)\dd x\sim \frac{yR(y)}{1+\rho}  \quad (y\to+\infty).$$
\end{lemma}
Recall that for an operator defined by
$(\K f)(x)=\int_0^1 K(x,y)f(y)\dd y$,
we have 
$$
\|\K\|_1\le M_1\qquad \textrm{and}\qquad \|\K\|_\infty\le M_\infty\, ,
$$ where
\begin{equation}
M_1:=\esssup_{y\in [0,1]}\int_0^1 |K(x,y)|\dd x,\quad M_\infty:=\esssup_{x\in [0,1]}\int_0^1 |K(x,y)|\dd y. \label{eq:majorant_op}
\end{equation}
If the kernel $K$ is symmetric for $x$ and $y$, then $M_1=M_\infty$. In this case and if $M_1=M_\infty<\infty$, then by the Riesz–Thorin interpolation theorem (cf. \cite[Theorem~2.2.14]{davies2007linear}, taking $p_0=q_0=1,p_1=q_1=\infty$), for all $p\in [1,+\infty]$, we have
\begin{equation}
    \|\K\|_p\le M_1=M_\infty. \label{eq:bound_Lp}
\end{equation}

We are now ready to prove the theorem. As mentioned above, we first prove the following convergence of operators.
\begin{lemma}Let $\rho\in(-1,0)$, then for any $p\in[1,\infty]$ and any $N\in\N$, the fomulas \eqref{eq:def_Kappa_N} and \eqref{eq:def_Kappa} define bounded operators $\K_N^{(\rho)}$ and $\K^{(\rho)}$ on $L^p(0,1)$. Moreover we have
\begin{equation}
    \lim_{N\to\infty}\|\K_N^{(\rho)}-\K^{(\rho)}\|_p=0 \label{eq:converge_operator}
\end{equation}
for any $p\in[1,\infty]$. \label{lem:convergence_operator}
\end{lemma}
\begin{proof}
Let $K_N^{(\rho)}:[0,1]^2 \rightarrow \R$ and $K_{\rho}:[0,1]^2 \rightarrow \R$ be the integral kernels defining respectively $\K_N^{(\rho)}$ and $\K^{(\rho)}$, that is,
$$K_N^{(\rho)}(x,y)=\frac{R(\lfloor Nx\rfloor-\lfloor Ny\rfloor)}{R(N)}, \quad K^{(\rho)}(x,y)=|x-y|^\rho.$$
Recall that since $R$ is even, the considered kernels are symmetric and the two essential supremums in \eqref{eq:majorant_op} of each kernel are equal. 
Moreover, for each $N$, $K_N^{(\rho)}$ is bounded on $[0,1]^2$ as it takes only a finite number of values, hence
$$
\esssup_{x\in [0,1]}\int_0^1 |K_N^{(\rho)}(x,y)|\dd y=\esssup_{y\in [0,1]}\int_0^1 |K_N^{(\rho)}(x,y)|\dd x <\infty\, .
$$
For $\rho\in(-1,0)$, easy calculations yield
$$\esssup_{x\in[0,1]}\int_0^1|x-y|^\rho\dd y=\esssup_{y\in[0,1]}\int_0^1|x-y|^\rho\dd x=\frac{2^{-\rho}}{(1+\rho)}<\infty\, .$$
So by \eqref{eq:bound_Lp}, for all $p\in [1,+\infty]$ we have
$$\|\K_N^{(\rho)}\|_p<\infty\quad\text{and }\quad\|\K^{(\rho)}\|_p\le \frac{2^{-\rho}}{(1+\rho)}\, .$$
Also by \eqref{eq:bound_Lp}, we have
\begin{equation}
\|\K_N^{(\rho)}-\K^{(\rho)}\|_p\ \le\ \esssup_{y\in[0,1]}\int_0^1\left|\frac{R(\lfloor Nx\rfloor-\lfloor Ny\rfloor)}{R(N)}-|x-y|^\rho\right|\dd x\, . \label{eq:bound_diff}
\end{equation}

We now prove that $\|\K_N^{(\rho)}-\K^{(\rho)}\|_p\to 0$ by showing that the RHS of \eqref{eq:bound_diff} goes to $0$ as $N\to\infty$. Taking an arbitrary $\varepsilon\in (0,1)$, we set $A_\varepsilon:=\{(x,y)\in[0,1]^2:|x-y|>\varepsilon\}$ and for $y\in[0,1]$, we set $A_\varepsilon(y):=\{x\in[0,1]:(x,y)\in A_\varepsilon\}=\{x\in[0,1]:|x-y|>\varepsilon\}$.

By the inequality
\begin{equation}
|x-y|-\frac{1}{N} \le \frac{|\lfloor Nx\rfloor-\lfloor Ny\rfloor|}{N}\le |x-y|+\frac{1}{N} \label{eq:ineq_intpart}
\end{equation}
and the uniform continuity of the function $x\mapsto x^\rho$ on $[\frac \varepsilon2,+\infty)$, we can take $N_1=N_1(\varepsilon)\in\N$ such that for $ N>N_1$ and 
$(x,y)\in A_\varepsilon$ we have
\begin{equation}
    \frac{|\lfloor Nx\rfloor-\lfloor Ny\rfloor|}{N}>\frac{\varepsilon}{2} \label{eq:ineq_A1_1}
\end{equation}
and 
\begin{equation}
    \left|\frac{|\lfloor Nx\rfloor-\lfloor Ny\rfloor|^\rho}{N^\rho}-|x-y|^\rho\right|<\frac{\varepsilon}{2}.\label{eq:ineq_A1_2}
\end{equation}
Then applying Lemma \ref{lem:unif_con} with $a=\varepsilon/2$, we can find $N_2=N_2(\varepsilon)\in\N$ such that for $N>N_2$ and for all $c$ satisfying $|c|\ge \varepsilon/2$, we have
\begin{equation}\left|\frac{R(cN)}{R(N)}-|c|^\rho\right|<\frac{\varepsilon}{2}. \label{eq:ineq_A1_3} \end{equation}
For all $N>\max(N_1,N_2)$ and $(x,y)\in A_\varepsilon$, let $c=\frac{\lfloor Nx\rfloor-\lfloor Ny\rfloor}{N}$ then by \eqref{eq:ineq_A1_1} we have $|c|>\varepsilon/2$. Moreover 
\begin{equation}
\left|\frac{R(\lfloor Nx\rfloor-\lfloor Ny\rfloor)}{R(N)}-\frac{|\lfloor Nx\rfloor-\lfloor Ny\rfloor|^\rho}{N^\rho}\right|<\frac{\varepsilon}{2} \label{eq:ineq_A1_4}
\end{equation}
by \eqref{eq:ineq_A1_3}.
Combining \eqref{eq:ineq_A1_2}, \eqref{eq:ineq_A1_4} and the triangle inequality, we have
$$\left|\frac{R(\lfloor Nx\rfloor-\lfloor Ny\rfloor)}{R(N)}-|x-y|^\rho\right|<\varepsilon$$
for all $N>\max(N_1,N_2)$ and $(x,y)\in A_\varepsilon$.
Then for $N$ large enough, we have
\begin{equation}
\esssup_{y\in[0,1]}\int_{A_\varepsilon(y)}\left|\frac{R(\lfloor Nx\rfloor-\lfloor Ny\rfloor)}{R(N)}-|x-y|^\rho\right|\dd x<\varepsilon\,. \label{eq:proof_toep_control1}
\end{equation}

On the other hand, for all $y\in[0,1]$, we have
\begin{equation}
\int_{[0,1]\backslash A_\varepsilon(y)}|x-y|^\rho\dd x\le \int_{-\varepsilon}^\varepsilon |x|^\rho\dd x=\frac{2\varepsilon^{1+\rho}}{1+\rho}. \label{eq:ineq_A2_1}
\end{equation}
Hence we just need to control the integral
$$\int_{[0,1]\backslash A_\varepsilon(y)}\left| \frac{R(\lfloor Nx\rfloor-\lfloor Ny\rfloor)}{R(N)}\right| \dd x\,.$$
Notice that both $R$ and $|R|$ are even, locally bounded and regularly varying with index $\rho$. By Lemma \ref{lem:equi_integ}, we have
\begin{eqnarray}
\int_{[0,1]\backslash A_\varepsilon(y)}\left| \frac{R(\lfloor Nx\rfloor-\lfloor Ny\rfloor)}{R(N)}\right| \dd x 
&\le& \int_{y-\varepsilon}^{y+\varepsilon}\left|\frac{R(\lfloor Nx\rfloor-\lfloor Ny\rfloor)}{R(N)}\right| \dd x \nonumber \\
    &\stackrel{(a)}=&\int_{-\varepsilon}^{\varepsilon}\left| \frac{R(\lfloor Nx+(Ny-\lfloor Ny\rfloor)\rfloor)}{R(N)}\right| \dd x \nonumber \\
    &=&\int_{-N\varepsilon}^{N\varepsilon}\left| \frac{R(\lfloor x+(Ny-\lfloor Ny\rfloor)\rfloor)}{NR(N)}\right| \dd x \nonumber\\
    &\le& \int_{-N\varepsilon-1}^{N\varepsilon+1}\left| \frac{R(\lfloor x\rfloor)}{NR(N)}\right| \dd x \nonumber \\
    & \stackrel{(b)}\sim& 2\, \left| \frac{R(N\varepsilon+1)}{R(N)}\right| \, \frac{\varepsilon}{1+\rho}\quad \stackrel{(c)}\sim\quad  \frac{2\varepsilon^{1+\rho}}{1+\rho} 
    \label{eq:ineq_A2_2}
\end{eqnarray}
as $N\to \infty$, where $(a)$ follows from a change of variable and the fact that $\lfloor x \rfloor + h = \lfloor x +h \rfloor$ for every $h\in \mathbb{Z}$, 
$(b)$ follows from Lemma \ref{lem:equi_integ} and $(c)$ from Lemma \ref{lem:unif_con}.
Notice that the controls \eqref{eq:ineq_A2_1} and \eqref{eq:ineq_A2_2} are independent of $y$, hence for $N$ large enough, we have
\begin{equation}
\esssup_{y\in[0,1]}\int_{[0,1]\backslash A_\varepsilon(y)}\left|\frac{R(\lfloor Nx\rfloor-\lfloor Ny\rfloor)}{R(N)}-|x-y|^\rho\right|\dd x\quad <\quad \frac{5\varepsilon^{1+\rho}}{1+\rho}\,. \label{eq:proof_toep_control2}
\end{equation}
Combining \eqref{eq:proof_toep_control1} and \eqref{eq:proof_toep_control2}, and taking $\varepsilon\to 0$, we finally obtain  
$$
\|\K_N^{(\rho)}-\K^{(\rho)}\|_p\xrightarrow[N\to\infty]{} 0
$$ 
for all $p\in[1,\infty]$.
\end{proof}
As a consequence of Lemma~\ref{lem:convergence_operator}, we conclude that $\K^{(\rho)}$ is compact on $L^p(0,1)$ for all $p\in[1,\infty]$, because it is the limit in operator norm of finite dimensional operators $\K_N^{(\rho)}$. 

We will complete the proof of Theorem~\ref{th:main_toepl} in the next section.

\subsection{Convergence of eigenvalues and simplicity of the largest eigenvalue}
First we note that $\K^{(\rho)}$, as an operator on $L^2(0,1)$, is self-adjoint and nonnegative definite. The definite nonnegativity can be concluded from the convergence \eqref{eq:converge_operator}. Indeed, taking the slowly varying function $L\equiv 1$ in the definition \eqref{eq:intro:long_memory} of $T_N$, by Polya's Theorem (see for example Theorem~3.5.22 of \cite{jacob2001pseudo}), the Toeplitz matrices $T_N$ are positive definite for all $N$. Thus by Widom-Shampine Lemma~\ref{lem:widom_shamp}, $\K^{(\rho)}_N$ does not have negative eigenvalues, since it has the same nonzero eigenvalues as $T_N/(N\gamma(N))$. Then for any $f\in L^2(0,1)$, we have
$$\langle \K_N^{(\rho)} f,\, f\rangle\ge 0\,.$$
Let $N\to\infty$, from \eqref{eq:converge_operator} we have
$$\langle \K^{(\rho)} f,\, f\rangle\ge 0\,.$$

For $k=1,2,\dots$, let $a_k^{(\rho)}$ be the $k$-th largest eigenvalue of $\K^{(\rho)}$. The asymptotic formula of $a_k^{(\rho)}$ as $k\to\infty$ has been obtained by Kac and Widom. See for example eq. (2) of \cite{widom1963asymptotic}. Thus $a_k^{(\rho)}>0$ for all $k\ge 1$.

Let $a^{(\rho)}_{N,k}=\lambda_k(T_N(c))/(NR(N))$ be the $k$-th largest eigenvalue of $\K_N^{(\rho)}$. From the convergence $\|\K_N^{(\rho)}-\K^{(\rho)}\|_2\to 0$ we deduce that $a^{(\rho)}_{N,k}\to a_k^{(\rho)}$ as $N\to\infty$. In fact, as $\K_N^{(\rho)}$ and $\K^{(\rho)}$ are compact and self-adjoint, by the Min-Max Formula (see also \cite[Theorem 4.12]{teschl2014mathematical}) we have
\begin{eqnarray*}
a^{(\rho)}_{N,k}&=&\min_{\dim U=k-1}\max_{\substack{u\in U^\perp \\ \|u\|_2=1}}\langle u,\K_N^{(\rho)}u\rangle\\
&\le& \min_{\dim U=k-1}\max_{\substack{u\in U^\perp \\ \|u\|_2=1}}\langle u,\K^{(\rho)} u\rangle+\|\K_N^{(\rho)}-\K^{(\rho)}\|_2\quad =\quad a_k^{(\rho)}+\|\K_N^{(\rho)}-\K^{(\rho)}\|_2\, .
\end{eqnarray*}
Symmetrically, we also have $a_k^{(\rho)}\le a^{(\rho)}_{N,k}+\|\K_N^{(\rho)}-\K^{(\rho)}\|_2$, from which we deduce
$$|a^{(\rho)}_{N,k}-a_k^{(\rho)}|\le \|\K_N^{(\rho)}-\K^{(\rho)}\|_2$$
for all $N$ and $k$. This implies the convergence of each eigenvalue $a^{(\rho)}_{N,k}$ toward $a_k^{(\rho)}$.

We now prove that $a_1^{(\rho)}$ is a simple eigenvalue of $\K^{(\rho)}$. Let $u\in L^2(0,1)$ be an eigenfunction of $a_1^{(\rho)}$, then by the mini-max formula $a_1^{(\rho)}=\max_{f\in L^2,\|f\|_2=1}\langle f,\K^{(\rho)} f\rangle$, we have
$$
a_1^{(\rho)}=\int_0^1\int_0^1 |x-y|^\rho u(x)\overline{u(y)}\dd x\dd y\le \int_0^1\int_0^1 |x-y|^\rho|u(x)u(y)|\dd x\dd y\le a_1^{(\rho)}
$$
which implies
$$\int_0^1\int_0^1 |x-y|^\rho(|u(x)u(y)|-u(x)\overline{u(y)})\dd x\dd y=0.$$
Hence $|u(x)u(y)|=u(x)\overline{u(y)}$ for $(x,y)\in[0,1]^2$ $\dd x\dd y$-a.e. This implies that for almost all $y\in[0,1]$, the equality $|u(x)u(y)|=u(x)\overline{u(y)}$ holds for almost all $x\in[0,1]$. Let $y_0$ be such that $u(y_0)\ne 0$ and $c=u(y_0)/|u(y_0)|=e^{i\varphi_0}$. Then for almost every $x\in[0,1]$, we have
$$u(x)=\frac{|u(x)u(y_0)|}{\overline{u(y_0)}}=c|u(x)|.$$
So up to a nonzero constant multiplier we can suppose that $u\ge 0$ on $[0,1]$. Therefore
$$a_1^{(\rho)} u(x)=\int_0^1 |x-y|^\rho u(y)\dd y\ge \int_0^1 u(y)\dd y>0\, .$$
This implies that $a_1^{(\rho)}>0$ and $u(x)>0$ for all $x\in[0,1]$. Then for any other function $v\in L^2(0,1)$ s.t. $\langle u,v\rangle =0$, $v$ cannot be an eigenfunction of the eigenvalue $a_1^{(\rho)}$. Otherwise following the same line of reasoning as previously we may write $v=|v|e^{\ii\tilde \varphi_0}$ where $|v|$ is also an eigenfunction associated to $a_1^{(\rho)}$, $|v|>0$ on $[0,1]$ and $\langle u,|v| \rangle = e^{-\ii\tilde \varphi_0} \langle u,v\rangle =0$. But $u,|v|>0$ contradict the orthogonality
$$
\langle u,|v|\rangle=\int_0^1 u(x)|v(x)|\dd x=0\, .
$$

Finally recall that if $R$ is regularly varying of index $\rho\in(-1,0)$, then $NR(N)\to\infty$ as $N\to\infty$ by \cite[Prop. 1.3.6(v)]{bingham1989regular}. Combining Widom-Shampine's lemma and the convergence of eigenvalues, we obtain
$$
a^{(\rho)}_{N,k}=\frac{\lambda_k(T_N(c))}{NR(N)} \xrightarrow[N\to\infty]{} a_k^{(\rho)}\, .
$$ 
Hence $\lambda_k(T_N(c))\sim a_k^{(\rho)} NR(N)\to\infty$ as $N\to\infty$ and 
$$
\frac{\lambda_2(T_N(c))}{\lmax(T_N(c))}\xrightarrow[N\to\infty]{} \frac{a_2^{(\rho)}}{a_1^{(\rho)}}\, <\, 1\, .
$$ 
Proof of Theorem~\ref{th:main_toepl} is now completed.

\bibliography{ref}
\bibliographystyle{plain}

%
%
%
%
\end{document}